\documentclass[a4paper,11pt]{amsart}
\usepackage{amsfonts,amsmath,amssymb,amsthm,mathrsfs}
\usepackage[latin9]{inputenc}
\usepackage{geometry}
\usepackage[colorlinks=true,pagebackref=false]{hyperref}
\usepackage{fixmath}
\usepackage{paralist}
\usepackage{graphicx}
\usepackage{verbatim}
\hypersetup{%
pdftitle={},
pdfsubject={},
pdfauthor={},
pdfkeywords={},
plainpages=false,pdfborder={0 0 0}}

\numberwithin{equation}{section}

\newtheorem{COUNT}{}[section]

\theoremstyle{plain}
\newtheorem{theorem}[COUNT]{Theorem}
\newtheorem{proposition}[COUNT]{Proposition}
\newtheorem{lemma}[COUNT]{Lemma}
\newtheorem*{lemma*}{Lemma}
\newtheorem{corollary}[COUNT]{Corollary}
\theoremstyle{definition}

\theoremstyle{remark}

\newcommand{\paragraphstartformat}[1]{\subsection*{#1}}

\def\skipassumptiona{0cm}
\def\skipassumptionb{1cm}
\def\skipassumptionc{0cm}

\def\assumption#1#2{\vskip \skipassumptiona \makebox[\skipassumptionb][l]{#1} \hskip \skipassumptionc #2 \vskip \skipassumptiona}

\renewcommand{\MR}[1]{}

\def\AAa{\mathcal{A}}
\def\BB{\mathcal {B}}

\def\DD{\mathcal {D}}
\def\EE{\mathcal {E}}
\def\FF{\mathcal {F}}
\def\GG{\mathcal {G}}
\def\HH{\mathcal {H}}

\def\LL{\mathcal {L}}
\def\MM{\mathcal {M}}

\def\RR{\mathcal {R}}
\def\SSs{\mathcal {S}}

\def\VV{\mathcal {V}}

\def\B{\mathbb {B}}
\def\D{\mathbb {D}}
\def\E{\mathbb {E}}
\def\F{\mathbb {F}}
\def\N{\mathbb {N}}
\renewcommand{\P}{\mathbb {P}} 
\def\Q{\mathbb {Q}}
\def\R{\mathbb {R}}
\def\T{\mathbb {T}}
\def\U{\mathbb {U}}

\DeclareMathOperator{\locc}{loc}
\DeclareMathOperator{\Laplace}{{\Delta}}
\def\loc{{\locc}}
\def\leb{\mathscr{L}}

\def\MP{\mathrm{MP}}
\def\id{\mathrm{I}}
\def\EDS{Let $(\Omega, \FF,\P,\tau)$ be an ergodic dynamical system.}

\def\hamil{\HH}
\def\lagrange{\LL}
\def\ess{\mathrm{ess}}
\def\pott{\textit{regular potential}}
\def\g{\kappa}
\def\modulo{\mathrm{mod\ }}
\def\spanning{\mathrm{span}}
\def\euclid{\mathrm{eu}}
\def\sym{\mathrm{Sym}}
\def\vmax{{v_{\mathrm{max}}}}

\def\div{\mathrm{div}}
\def\rot{\mathrm{rot}}
\def\pot{\mathrm{pot}}
\def\sol{\mathrm{sol}}
\def\FFF{\mathfrak{F}}
\def\DDD{\mathfrak{D}}

\def\QQQ{\mathfrak{Q}}
\def\PPPhi{\mathfrak{P}}
\def\subw{\mathrm{w}}
\def\supw{\mathrm{w}}
\def\subs{\mathrm{s}}
\def\sups{\mathrm{s}}
\def\compactc{\mathrm{c}}
\def\complemetc{\mathrm{c}}

\begin{document}

\title[A Variational Formula for the Lyapunov Exponent]{A Variational Formula for the Lyapunov Exponent of Brownian Motion
in Stationary Ergodic Potential}
\author{Johannes Rue\ss}
\address{
Mathematisches Institut, Eberhard Karls Universit\"at T\"ubingen\\
Auf der Morgenstelle 10, 72076 T\"ubingen, Germany
}


\begin{abstract}
We establish a variational formula for the exponential decay rate of the Green function of Brownian motion evolving in a random stationary and ergodic nonnegative potential. Such a variational formula is established by Schroeder in \cite{Schroeder88} for periodic potentials and is generalised in the present article to a non-compact setting. We show exponential decay of the Green function implicitly. This formula for the Lyapunov exponent has several direct implications. It allows to compare the influence of a random potential to the influence of the averaged potential. It also leads to a variational expression for the quenched free energy.
\end{abstract}

\maketitle

\section{Introduction and Results}

Decay of the Green function for Brownian motion has been subject of study in many respects. If Brownian motion is evolving in a random potential, additional properties of the random structure of the potential such as ergodicity and stationarity allow to expect that the Green function exhibits a deterministic behaviour on the large scale. A deterministic exponential decay rate of the Green function, also called Lyapunov exponent, has been already established in many cases. In this article we give a variational formula for the Lyapunov exponent of Brownian motion in a stationary ergodic potential. Such a variational expression has been proven by Schroeder in \cite{Schroeder88}. Schroeder considers periodic potentials and therefore deals with potentials defined on compact spaces. In the present work we generalise the results there to the non-compact setting of stationary ergodic potentials.

We consider Brownian motion on $\R^d$. Let $Z=(Z_t)_{t\geq 0}$ be the canonical process on the space $C([0,\infty),\R^d)$, $d \in \N$, with the cylindrical $\sigma$-field $\SSs$. Let $P_x^\lambda$ denote the law on $\SSs$ of a $d$-dimensional Brownian motion with drift $\lambda \in \R^d$ and starting at $x \in \R^d$, let $E_x^\lambda$ denote the expectation under $P_x^\lambda$. If $x$, $\lambda=0$ we write $P$ and $E$ instead of $P_x^\lambda$ and $E_x^\lambda$ respectively.

The Brownian motion is assumed to evolve in a random potential defined on a probability space $(\Omega,\FF,\P)$: Let $\sym(\Omega)$ be the symmetric group on $\Omega$. We assume that $(\R^d,+)$ is acting on $\Omega$ via a homomorphism $\tau: \R^d \to \sym(\Omega)$, $x \mapsto \tau_x$, such that
$(x,\omega) \mapsto \tau_x\omega$ is a product measurable mapping and $\tau_x$ is measure preserving.
Ergodicity of $\P$ with respect to the family of transformations $\{\tau_x:\, x \in \R^d\}$ will be crucial and assumed additionally in many cases. Then $(\Omega,\FF,\P,\tau)$ is said to be an ergodic dynamical system. We denote the space of integrable functions on $\Omega$ by $L^1$. A function $V \in L^1$ that is non-negative is called \textit{potential} throughout the article. For any $\omega \in \Omega$ we may consider the realisation of $V$ as a function on $\R^d$ via $V_\omega(x):=V(\tau_x\omega)$. In order to avoid trivialities we assume for $\omega \in \Omega$ that the realisation $V_\omega$ of a potential $V$ is not negligible with respect to the Lebesgue measure.

We investigate the Green function for Brownian motion evolving in $V_\omega$. Define
\begin{align}\label{e:def_green}
G(x,A,\omega)
&:= E_x \bigg[\int_0^\infty \exp \bigg\{-\int_0^t V_\omega(Z_s)ds \bigg\}1_A(Z_t)dt\bigg],
\end{align}
where $x \in \R^d$, $\omega \in \Omega$ and $A\in\BB(\R^d)$, the Borel $\sigma$-algebra of $\R^d$. $G$ can be interpreted as the expected occupation times measure of Brownian motion killed at rate $V_\omega$, and is also called the Green measure. In this article we are generally in the situation that for $x \in \R^d$ and $\omega \in \Omega$ the Green measure possess a continuous density $g(x,\cdot,\omega):$ $\R^d\setminus\{x\} \to  [0,\infty)$ with respect to the Lebesgue measure, which is called Green function, see condition (G) below. Under natural assumptions the Green function can be interpreted as the fundamental solution to $-(1/2)\Laplace + V_\omega$, that is
\begin{align*}
\bigg(-\frac{1}{2}\Laplace+V_\omega \bigg)g(x,\cdot,\omega) = \delta_x,
\end{align*}
where $\delta_x$ denotes the Dirac distribution at $x \in \R^d$, see \cite[Theorem 4.3.8]{Pinsky1995}.

If the Green function decays exponentially fast with a deterministic exponential decay rate, which means that for $y\in \R^d\setminus\{0\}$ the limit
\begin{align}\label{eq:def_lyapunovexponent}
\alpha_V(y):= \lim_{r \to \infty} - \frac{1}{r} \ln g(0,ry,\omega)
\end{align}
exists and is $\P$-a.s.\ constant, then the Lyapunov exponent is said to exist and is defined as $\alpha_V$. Conventionally we set $\alpha_V(0):=0$.

Existence of the Lyapunov exponent is shown for example for Poissonian potentials by Sznitman in \cite[Theorem 0.2]{Sznitman94}. We refer to \cite[Chapter 5]{Sznitman98} where an overview can be found. Results for a discrete space counterpart are given by Zerner in \cite[Theorem A]{Zerner98} and extended by Mourrat in \cite[Theorem 1.1]{Mourrat11}. The operator $-(1/2)\Laplace + V$ is also called a random Schrödinger operator. A comprehensive treatise on the theory of random Schrödinger operators can be found e.g.\ in \cite{Stollmann01} or \cite{Lacroix1990}. Typically \eqref{eq:def_lyapunovexponent} is proven by means of subadditivity. Existence of the Lyapunov exponent is part of our main Theorem \ref{t:varform} and we do not need the subadditive ergodic theorem.

\paragraphstartformat{Variational formula in the case of periodic potential.} In \cite{Schroeder88} Schroeder shows exponential decay of the Green function and establishes a variational formula for the Lyapunov exponent of Brownian motion in periodic potential. We recall the result of Schroeder: Let $\Omega=\T^d$ be the $d$-dimensional torus equipped with the Borel $\sigma$-algebra and Lebesgue measure. For $x \in \R^d$ and $\omega \in \Omega$ define $\tau_x\omega:= \omega+x \;(\modulo{} 1)$. We get an ergodic dynamical system on which realisations of potentials are periodic. Let $C^k(\T^d)$ be the space of continuous functions on $\T^d$ with continuous derivatives of order less than or equal to $k$ and let $y \in \R^d\setminus\{0\}$. The variational expression given by Schroeder varies over probability densities $f \in C^2(\T^d)$ such that $f>0$ and $\int f dx = 1$, and over divergence-free vector fields $\phi \in (C^1(\T^d))^d$ such that 
$\int_{\T^d} \phi(x) \, dx = y$ and $\nabla \cdot \phi = 0$.
The following is shown in \cite{Schroeder88} and is generalised in our main Theorem \ref{t:varform}:

\begin{theorem}{\cite[(1.1)]{Schroeder88}}\label{t:varform_periodic}
Let $V$ be a continuous potential on $\Omega=\T^d$ such that $V(\omega)>0$ for $\omega \in \Omega$. Then for all $\omega \in \Omega$, $y \in \R^d \setminus \{0\}$,
\begin{align}\label{eq:varform_periodic}
\lim_{r \to \infty}-\frac{1}{r}\ln g(0,ry,\omega)
= 2\inf_f \left[ \left( \int_{\T^d} \frac{|\nabla f|^2}{8f} + Vf \,dx\right) \left(\inf_\phi \int_{\T^d}\frac{|\phi|^2}{2f} \, dx\right) \right]^{1/2}.
\end{align}
\end{theorem}

\paragraphstartformat{Variational formula for more general potentials.}
A natural way to generalise periodicity is to study stationary ergodic potentials. Different kinds of stationary ergodic potentials have been considered in the literature. For example random chessboard potentials are described in \cite[(3.4)]{DalMaso1986}. A huge class of stationary ergodic potentials is given by potentials that are generated by random measures:

Let $\Omega$ be the set $\MM(\R^d)$ of locally finite measures on $(\R^d,\BB(\R^d))$ equipped with the topology of vague convergence, and let $\FF$ be the associated Borel $\sigma$-algebra on $\Omega$. For $\omega \in \Omega$, $x \in \R^d$, $A \in \BB(\R^d)$ introduce $\tau_x\omega[A] := \omega[A+x]$, and let $\P$ be the distribution on $\Omega$ of a stationary and ergodic random measure with state space $\R^d$. Then $(\Omega,\FF,\P, \tau)$ becomes an ergodic dynamical system, use e.g.\ \cite[Exercise 12.1.1(a)]{Daley2008}. Given a measurable `shape function' $W: \R^d \to \R_{\geq 0}$ we can define the potential generated by the random measure $\P$ by $V: \Omega \to [0,\infty]$,
\begin{align*}
V (\omega) := \int W(x) \omega(dx).
\end{align*}
Truncation leads to bounded potentials, and with the help of convolution $V$ can be equipped with regularity, see Appendix \ref{app:derivative}. Often studied examples of this type are potentials with an underlying Poisson point process $\P$, also called Poissonian potentials, see e.g.\ \cite{Sznitman98}.

In order to formulate the variational expression in \eqref{eq:varform_periodic} for a general non-compact stationary setting we need some more notation. If $f_\omega$ is differentiable we write $Df(\omega)$ for the derivative $Df_\omega(0)$. We introduce a space of probability densities, and for $y \in \R^{d}$ we introduce a space of measurable divergence-free vector fields:
\begin{align*}
\F_\subs:=&\{ f \in L^1: \; 
f_\omega \text{ is differentiable of any order for all } \omega \in \Omega,\\
& \hspace{2.5cm} \E[f]=1,\; \exists \; c_f>0 \text{ s.t.} \; f \geq c_f,\;
\sup_\Omega|D^nf|<\infty \text{ for any } n \in \N_0\}.\\
\Phi_y^\sups:=&\{\phi \in (L^1)^d:  \;  \phi_\omega \text{ is differentiable of any order for all } \omega \in \Omega,\\
& \hspace{2.5cm} \E[\phi]= y,\; \nabla \cdot \phi = 0,\; 
\sup_\Omega|D^n\phi|<\infty \text{ for any } n \in \N_0\}.
\end{align*}
Superscript `s' here emphasises that spaces of `strongly' differentiable functions are considered in contrast to spaces of `weakly' differentiable functions introduced in Section~\ref{sec:preliminaries}.

In analogy to \eqref{eq:varform_periodic} we show in Theorem \ref{t:varform} that under sufficient conditions the Lyapunov exponent can be represented as the variational expression
\begin{align}\label{eq:definition_varform}
\Gamma_V(y)
:=2 \inf_{f \in \F_\subs} \left[\left( \int\frac{|\nabla f|^2}{8f} + Vf d\P \right) \left( \inf_{\phi \in \Phi_y^\sups} \int \frac{|\phi|^2}{2f} d\P\right)\right]^{1/2}.
\end{align}

\paragraphstartformat{Lyapunov exponent and the quenched free energy.}
On the way we obtain a representation of the Lyapunov exponent in terms of the quenched free energy: For $y \in \R^d\setminus\{0\}$ $\P$-a.s.,
\begin{align} \label{eq:alpha_intermsof_qfe}
\alpha_V(y)=\RR(-\Lambda_\omega)(y).
\end{align}
Here,
\begin{align}\label{eq:defQFE}
\Lambda_\omega(\lambda)
&:= \limsup_{t \to \infty} \frac{1}{t} \ln E^\lambda \bigg[\exp\bigg\{-\int_0^t V_\omega(Z_s)ds\bigg\} \bigg]
\end{align}
is the quenched free energy, and the functional $\RR$ is given as $\RR(a):$ $\R^d \to \R\cup\{\pm \infty\}$,
\begin{align*}
\RR(a)(y):= \sup\{\langle y, \lambda \rangle:\, \lambda \in \R^d,\, \lambda^2/2<a(\lambda) \},
\end{align*}
with $\sup\emptyset := -\infty$, and where $a$ is any real-valued function on $\R^d$. The functional $\RR$ also appears in the article \cite{Armstrong2012} of Armstrong and Souganidis, see \eqref{eq:identification} below. Note that we only look at the limit superior in \eqref{eq:defQFE} since this suffices for our purposes. For our results we also do not need to know whether $\Lambda_\omega$ is deterministic. Existence of the deterministic limit in \eqref{eq:defQFE} is shown in many cases and an overview over the literature can be found after Corollary \ref{cor:quenchedfreeenergy}. Often, this can be deduced via homogenization, see Appendix~\ref{app:homogenization}.

Equality \eqref{eq:alpha_intermsof_qfe} also allows to give a variational expression for the quenched free energy: Let
\begin{align*}
L^\lambda = \frac{1}{2}\Laplace + \lambda \cdot \nabla
\end{align*}
be the generator of Brownian motion with constant drift $\lambda \in \R^d$. We introduce
\begin{align*}
\sigma(\lambda)
&:= -\sup_{f \in \F_\subw^2} \inf_{u \in \U} \int \left(\frac{L^\lambda u}{u} - V\right) f d\P \geq 0,
\end{align*}
where $\F_\subw^2$ is a certain space of probability densities on $\Omega$, and $\U$ is a space of positive functions on $\Omega$. These function spaces are introduced rigorously in Section \ref{sec:preliminaries}. The representation
\begin{align}\label{eq:quenchedfreeenergy0}
\Lambda_\omega(\lambda)=-\sigma(\lambda)
\end{align}
is given in Corollary \ref{cor:quenchedfreeenergy}. Quantities like $\sigma$ appear naturally in many cases as representation of the principle Dirichlet eigenvalue for Markov processes, see e.g.\ \cite{Donsker75VariationalFormula}.

On a heuristic basis the following representations of $\sigma(0)$ and $\Gamma_V$ may be derived: We will see in Proposition \ref{prop:formulaforI} and \eqref{eq:sigma_independent_ofchoiceof_F}, that
\begin{align*}
\sigma(0)=\inf_{g^2 \in \F_\subs} \int \frac{|\nabla g|^2}{2} + V g^2 d\P.
\end{align*}
Let $\QQQ(\cdot,\cdot)$ denote the quadratic form associated to the generator of the Markov process $\omega_t=\tau_{Z_t}\omega$ on $\Omega$, $t \geq 0$, omitting existence issues and well-definedness. Then  $\QQQ(g,g)$ should equal $\int|\nabla g|^2/2 + Vg^2 d\P$ for $g \in \F_\subs$, see \cite[(1.4.18)]{Sznitman98}. Thus, we obtain
\begin{align*}
\sigma(0)=\inf_{f \in \F_\subs} \QQQ(\sqrt f, \sqrt f),
\end{align*}
which is analogous to the formula for the quenched free energy of Brownian motion in nonnegative deterministic potential given e.g.\ in \cite[(3.1.2)]{Sznitman98}.
These considerations also allow the following heuristic reformulation of $\Gamma_V$: For $y \in \R^d$,
\begin{align*}
\Gamma_V(y) = 2 \inf_{f \in \F_\subs} \left[ \QQQ(\sqrt f,\sqrt f) \cdot \inf_{\phi \in \Phi_y^\sups} \int \frac{|\phi|^2}{2f} d\P \right]^{1/2}.
\end{align*}

\paragraphstartformat{Assumptions.}
The following hypotheses are imposed from time to time:

{\parindent=0pt
\assumption
{(B)}{$\vmax := \sup_\Omega V < \infty$.}

\assumption
{(G)}
{$g(0,\cdot,\omega) \in C^2(\R^d \setminus \{0\})$ for $\omega \in \Omega$,\newline
\makebox[\skipassumptionb][l]{} \hskip \skipassumptionc $L_\omega g(0,\cdot,\omega) = 0$ on $\R^d \setminus \{0\}$, where $L_\omega := (1/2)\Laplace-V_\omega$, for $\omega \in \Omega$.}

\assumption{(E1)}
{For any $\lambda \in \R^d$ such that $|\lambda|^2/2 < \sigma(\lambda)$ one has $\P$-a.s.,
$
\sigma(\lambda) 
\leq -\Lambda_\omega(\lambda).
$}

\assumption{(E2)}{$\sigma(0)>0$.}
}
There always exists a density for the Green measure with respect to the Lebesgue measure in the present context, see e.g.\ \cite[(2.2.3)]{Sznitman98}. Condition (G) holds under weak regularity assumptions on the potential, such as local Hölder continuity, see \cite[Theorem 4.2.5(iv)]{Pinsky1995}.
There is a broad variety of cases in which condition (E1) is known to be valid: In Appendix \ref{app:homogenization} we show that in the stationary ergodic case, the expression of the effective Hamiltonian given by Kosygina, Rezakhanlou and Varadhan in \cite{Kosygina2006} guarantees (E1). Condition $\sigma(0)>0$ is related to the question whether the effective Hamiltonian at zero is negative. Answers to this question are given in \cite[Proposition 5.9]{Armstrong2012}. If $V$ is strictly bounded away from zero, then $\sigma(0)>0$ trivially. A potential satisfying (B), (G), (E1), (E2) is referred to as a \textit{\pott{}} throughout this article.

\section*{Results}

\paragraphstartformat{Variational formula}

Our main result is the following variational formula and implicitly shows that the Green function decays exponentially fast with a deterministic decay rate, that is, that the Lyapunov exponent exists for Brownian motion evolving in a stationary ergodic potential. Moreover, it expresses the Lyapunov exponent in terms of $\sigma$ and $\Lambda_\omega$:

\begin{theorem}\label{t:varform}
Assume $(\Omega,\FF,\P,\tau)$ is an ergodic dynamical system and $V$ is a regular potential. Let $y \in \R^d \setminus \{0\}$, then $\P$-a.s.\ the limit 
\begin{align*}
\alpha_V(y):=\lim_{r \to \infty} - \frac{1}{r}\ln g(0,ry,\omega)
\end{align*}
exists and $\P$-a.s.\ one has the variational expression
\begin{align}\label{eq:varform}
\alpha_V(y) = \Gamma_V(y).
\end{align}
Moreover, $\P$-a.s.\ $\RR(-\Lambda_\omega) = \Gamma_V = \RR(\sigma)$.
\end{theorem}

Theorem \ref{t:varform} is a generalisation of the variational formula for Lyapunov exponents of Brownian motion in periodic potentials established by Schroeder in \cite{Schroeder88} to stationary and ergodic potentials. We extend the techniques developed by Schroeder to prove Theorem \ref{t:varform}. A refined examination allowed us to express the Lyapunov exponent in terms of $\sigma$ and $\Lambda_\omega$.

In order to show existence of the Lyapunov exponent, in \cite{Sznitman94} the potential is assumed to satisfy a finite range dependence property and in \cite{Zerner98, Mourrat11} the potential is assumed to be i.i.d.. In this respect, the present work generalises the existence of the Lyapunov exponent to potentials with long range dependencies.

Some conditions may be relaxed if only parts of the results are considered. For example for $\alpha_V \leq \Gamma_V$ we only need to impose (B) and (G), while for $\alpha_V \geq \RR(\sigma)$ we require (G) and (E1). For $\Gamma_V=\RR(\sigma)$ we only need (E2) and $V \in L^2$. For $\Gamma_V=\RR(\sigma)$ no ergodicity of $\P$ is required.

We outline alternative possible choices for the function spaces $\F_\subs$ and $\Phi_y^\sups$ in Proposition \ref{prop:different_spaces}. We derive further variational expressions for the Lyapunov exponent in Proposition \ref{t:alternativerep1} and Proposition \ref{prop:finalequality}. In Proposition \ref{prop:ga_sol_of_minimizationproblem} we show that $\alpha_V$ is the unique solution to a variational problem.

The non-compactness of the underlying probability space induced additional complexity. For example our upper bound relies decisively on the quite general and relatively new version of the Ergodic Theorem \cite[(A.9)]{DelTenno2009SpecialExamples}, which may be interpreted as Ergodic Theorem for the `point of view of the particle'. Formula \eqref{eq:quenchedfreeenergy0} shown by Donsker and Varadhan in the compact case and used by Schroeder in \cite[(3.7)]{Schroeder88}, is not applicable in the present setting. We applied homogenization results from \cite{Kosygina2006} to derive the required estimate, see Proposition \ref{prop:logmomgenfuncupperbound}. Additionally, the last density statement given in Lemma \ref{lem:spacesdense} which is essential for Proposition \ref{t:alternativerep1}, while obvious in the compact case, is not trivial in the present general setting. It is crucial for the transition to the non-compact setting to establish this density property.

Theorem \ref{t:varform} shows to be suitable to study the Lyapunov exponent in more detail. It enables for example to derive continuity properties of the Lyapunov exponent with respect to the potential and the underlying probability measure. It also allows to establish strict inequalities. Results into this direction are part of a subsequent article, see \cite{Ruess2014}.

\paragraphstartformat{Influence of randomness}

The variational formula given in Theorem \ref{t:varform} allows to determine the influence of the randomness of the potential on the Lyapunov exponent. Choosing $f \equiv 1$ and $\phi \equiv y$ in \eqref{eq:definition_varform} one has

\begin{corollary}\label{cor:comparisonwithaveragedpotential}
Let $V$ be a potential. Then $\Gamma_V \leq \Gamma_{\E V}$. Assume additionally that $(\Omega,\FF,\P,\tau)$ is an ergodic dynamical system and $V$ is a regular potential, then for $y \in \R^d$ $\P$-a.s.,
\begin{align}\label{eq:annealedbound}
\alpha_V(y) \leq \alpha_{\E V}(y).
\end{align}
\end{corollary}

This inequality has been derived for Poissonian potentials by the author in \cite[Theorem 4]{Ruess2012}. In the discrete setting of random walk in i.i.d.\ random potential such a result is available in a more general formulation more directly as proven by Zerner in \cite[Proposition 4]{Zerner98} with the help of Jensen's inequality. Note, that in the continuous setting a direct application of Jensen's inequality is not possible. In \cite{Ruess2012} a discretisation technique is applied; alternatively a generalisation of Jensen's inequality to a functional analytic framework might also lead to this result if one had an representation as in \eqref{eq:identification}, see the comment given in \cite{Ruess2012}.


\paragraphstartformat{Long range survival probabilities}

Let $y \in \R^d$ and $B(0,1)$ denote the open ball with centre $0$ and radius $1$ and $\overline{B(0,1)}$ its closure. We introduce
\begin{align*}
e(y,\omega):=E_y \bigg[\exp\bigg\{-\int_0^{H(0)}V_\omega(Z_s)ds\bigg\},\, H(0)<\infty \bigg],
\end{align*}
where $H(0):=\inf\{s > 0: Z_s \in \overline {B(0,1)}\}$ is the hitting time of $\overline{B(0,1)}$. Under suitable assumptions the results given in \cite{Armstrong2012} together with Theorem \ref{t:varform} show that the exponential decay of the Green function coincides with the exponential decay of $e(\cdot,\omega)$, see Proposition \ref{prop:identification}: For $y \in \R^d$ $\P$-a.s.,
\begin{align}\label{eq:identification}
\lim_{r \to \infty} -\frac{1}{r}\ln e(-ry,\omega) 
= \lim_{r \to \infty} -\frac{1}{r}\ln g(0,ry,\omega).
\end{align}
This is known already e.g.\ for Poissonian potentials, see \cite[Theorem 0.2]{Sznitman94}.

\paragraphstartformat{Quenched free energy}

By Theorem \ref{t:varform} a computation of the inverse of $\RR$ leads to the following.

\begin{corollary}
\label{cor:quenchedfreeenergy}
\EDS{} Let $V$ be a potential such that $V+\mu$ is a \pott{} for $\mu>0$. Choose $c \in \R$ such that $\P$-a.s.\ $c \leq \min\{\sigma(0),-\Lambda_\omega(0)\}$. Then $\P$-a.s.\ for those $\lambda \in \R^d$, which satisfy $\sigma(\lambda)-\lambda^2/2 <c$ or $-\Lambda_\omega(\lambda)-\lambda^2/2<c$, we have
\begin{align}\label{eq:quenchedfreeenergy}
\sigma(\lambda)=-\Lambda_\omega(\lambda).
\end{align}
\end{corollary}

Since $-\Lambda_\omega(0) \geq 0$ and $\sigma(0) \geq 0$ we can always choose $c = 0$. Note that for bounded potentials $\sigma(\lambda)-\lambda^2/2$ as well as $-\Lambda_\omega(\lambda)-\lambda^2/2$ tend to $-\infty$ as $|\lambda| \to \infty$, see Subsection~\ref{subsec:equalityonwholeR}.

Equality \eqref{eq:quenchedfreeenergy} is known if $\Omega$ is compact, we refer to the articles of Donsker and Varadhan \cite{Donsker75VariationalFormula, Donsker75b}. In the case of non-compact $\Omega$ the investigation is much more complicated, as the examples given in \cite[Paragraph 9]{Donsker76} illustrate. For Polish $\Omega$ the non-compact setting is studied again by Donsker and Varadhan in \cite{Donsker76,Donsker1983}. There, $L^\lambda$ is assumed to generate a Feller process on $\Omega$ and the existence of special subsolutions to $(L^\lambda u)/u \leq A$ for some $A \in \R$ is needed. On $\R^d$ \eqref{eq:quenchedfreeenergy} is investigated in \cite{Donsker1976PrincipalEigenvalue, Donsker1975AsymptoticEvaluationWienerIntegrals} by the same authors. Discrete space models have also been examined. For example for random walks in random potential such formulae are given by Rassoul-Agha, Seppalainen and Yilmaz in \cite{Rassoul11QuenchedFreeEnergy} in quite generality. We also refer to the detailed overview of literature given in \cite[Subsection 1.3]{Rassoul11QuenchedFreeEnergy} concerning this topic.

To the best of our knowledge the representation \eqref{eq:quenchedfreeenergy} of the quenched free energy is new in the present context. One might want to compare our results with the series of equations given in \cite[p.1497]{Kosygina2006}, but the exchange of infimum and supremum there is done heuristically and not rigorous. Note also, that the regularity of test-functions in the continuous space setting obstructs the direct application of a finite $\sigma$-algebra approach as e.g.\ used in \cite{Rassoul11QuenchedFreeEnergy}.

\paragraphstartformat{Organisation of this article.}
Section \ref{sec:preliminaries} contains notation and preliminary results. In particular a list of function spaces and some denseness results are given. In Section \ref{sec:proofofvarform} we prove Theorem \ref{t:varform}. Further representations of $\sigma$ and $\Gamma_V$ are given in Subsection \ref{subsec:equivalence}. Appendix \ref{app:appendix} contains additional material such as conditions under which (E1) is valid and proofs for \eqref{eq:identification}, Corollary \ref{cor:quenchedfreeenergy} as well as for the denseness results.

\section{Notation and Preliminaries}\label{sec:preliminaries}

In this section we present notation and some basic results. By $\BB(\R^d)$ we denote the Borel $\sigma$-algebra on $\R^d$ and with $\leb$ the Lebesgue measure on $\BB(\R^d)$. By $|\cdot|$ we mean the Euclidean norm in $\R^d$. Unit vectors in coordinate directions are denoted by $e_i$. The scalar product between vectors $x$, $y \in \R^d$ is sometimes denoted by $\langle x, y\rangle$, where $\langle \cdot, \cdot\rangle$ is the scalar product on $(L^2)^d$, see below. $S^{d-1}$ is the sphere in $\R^d$, $B(x,r)$ is the open ball in $\R^d$ around $x$ of radius $r$. We write $C_\compactc^\infty$ for the set of real-valued functions on $\R^d$ with compact support having derivatives of any order. $L^p_\loc$ denotes the space of locally $p$-integrable functions on $\R^d$. If $A$ is a subset of a topological space $X$, by $\overline A$ we mean the closure of $A$ and by $\partial A$ the boundary of $A$. For $x$, $y \in \R^d$ we say $x$ is parallel to $y$ if $x y =|x||y|$. For bounded measurable $b:\R^d\to\R^d$ we denote by $P^b_0$ the unique solution to the martingale problem $\MP(b,\id)$ starting at $\delta_0$, see \cite[Proposition 5.3.6, Corollaries 5.3.11, 5.4.8]{Karatzas1991}. The corresponding expectation operator is denoted by $E_0^b$.

For $p \geq 1$ we denote by $L^p$ the space of $p$-integrable functions on $(\Omega,\FF,\P)$. By $(L^p)^d$, $1 \leq p \leq \infty$, we denote the vector space of measurable $f:\Omega \to \R^d$ such that $\|f\|_p < \infty$ where $\|f\|_p:=(\sum_{i=1}^d \|f_i\|_p^p)^{1/p}$ in the case $p<\infty$ and $\|f\|_\infty:= \max_i \|f_i\|_\infty$ if $p=\infty$. $(L^2)^d$ is provided with the scalar product $\langle \phi, \psi \rangle := \E[\phi\cdot\psi]$, $\phi$, $\psi \in (L^2)^d$. If $f\in L^1$ is essentially bounded such that $f \geq c$ for some $c>0$, then we consider sometimes the inner product $\langle \phi,\psi \rangle_f:=\E[\phi \cdot \psi f]$ and the associated norm $\|\phi\|_f:= \E[\phi^2f]^{1/2}$ for $\phi$, $\psi \in (L^2)^d$.

Let $f$ be a measurable function on $\Omega$. If $\P$-a.s.\ the realisations of $f$ are differentiable, we say that $f$ is (classically) differentiable. The following concept of weak differentiability on $\Omega$ is used throughout the article: Denote for the moment by $D^{\euclid,\nu}u$ the $\nu$-th weak derivative of a function $u \in L^1_\loc$ (if it exists), where $\nu \in (\N_0)^d$ is a multi-index. Here we write superscript `$\euclid$' for `Euclidean' space derivative. $f$ is said to possess a $\nu$-th weak derivative if for $\P$-a.e.\ $\omega \in \Omega$ the realisation $f_\omega$ possesses $\nu$-th weak derivative and if there is a measurable function $g$ defined on $\Omega$ such that $\P$-a.s.\ $\leb$-a.e.\ $g_\omega(x) = D^{\euclid,\nu}(f_\omega)(x)$. $g$ is called the $\nu$-th weak derivative of $f$ and we denote it by $D^\nu f$. If $f$ possesses $\nu$-th weak derivative for any $\nu$ such that $0 \leq \nu_1 + \nu_2 + \ldots + \nu_d \leq m$, $m \in \N$, we say $f$ is weakly differentiable of order $m$. If $f$ is weakly differentiable of order one we simply call $f$ weakly differentiable. 

For $f$ weakly and classically differentiable, the classical derivative coincides with the weak derivative $\P$-a.s.\ and hence, without ambiguity we use the same symbols for the weak and the classical (partial) differential operators.
Partial derivatives are also denoted by $\partial_if$ and $\partial_{ij}f:=\partial_i\circ\partial_jf$. As for classically differentiable functions we introduce for weakly differentiable functions the gradient operator $\nabla$ and for functions for which $\partial_{ii}$, $i=1, \ldots d$, exist we introduce the Laplace operator $\Laplace$.

For $1 \leq i \leq d$ the shift defines a strongly continuous one-parameter group of unitary operators $S^i = (S_t^i)_{t \in \R}$ on $L^2$ via $S_t^i:$ $L^2 \to L^2$, $f \mapsto f \circ \tau_{te_i}$, see e.g.\ \cite[(7.2)]{Jikov1994}, whose generator is given by $(\partial_i,\DD(\partial_i))$ defined on
\begin{align*}
\DD(\partial_i):= \{f \in L^2: \, f \text{ is weakly differentiable in direction $i$},\, \partial_i f \in L^2 \}.
\end{align*}
Moreover, skew-adjointness of the generator shows for $f, g \in \DD(\partial_i)$ integration by parts:
\begin{align}\label{eq:partialintegration}
\E[f \partial_i g] = - \E[(\partial_i f) g] \text{ and } \E[\partial_i f]= 0.
\end{align}
These results correspond to \cite[Lemma A.4]{Schmitz2009} in the case $d=1$, and can be proven as in the Euclidean case, see e.g.\ \cite[II.2.10]{Engel2000}. More details can be found in the previous version \cite{Ruess2013} of this article.

Different function spaces will be needed throughout this article: We define $\|f\|_\nabla := \|f\|_2 + \sum_{i=1}^d \| \partial_i f\|_2$ for $f \in \bigcap_{i=1}^d \DD(\partial_i)$,
and we introduce spaces of test functions:
\begin{compactitem}[]
\item $\D_\subw := \bigcap_{i=1}^d \DD(\partial_i)$, 
\item $\D_\subw^2 := \{f \in \D_\subw :\, \partial_i f \in \DD(\partial_i)
\text{ for } i=1,\ldots d,\, \|f\|_\infty,\, \|\nabla f\|_\infty,\, \|\Laplace f\|_\infty < \infty\}$,
\item $\D_\subs := \{f \in \D_\subw :\, f_\omega \in C^\infty(\R^d)\, \forall \omega \in \Omega,\, \sup_\Omega |D^n f|<\infty \text{ for } n \in \N_0\}$.
\end{compactitem}
Exponential of space of test functions: $\U := e^{\D_\subw^2}$, using \cite[(7.18), Lemma 7.5]{Gilbarg83},
\begin{compactitem}[]
\item $\U = \{f \in \D_\subw^2 :\, \exists\, c>0 \text{ s.t. } f>c\, \P \text{-a.s.}\}$.
\end{compactitem}
Spaces of probability densities:
\begin{compactitem}[]
\item $\F_\subw := \{f \in \D_\subw :\, \E f = 1,\, \exists\, c>0 \text{ s.t. } f>c\, \P\text{-a.s.},\, \|f\|_\infty,\, \|\nabla f\|_\infty < \infty\}$,
\item $\F_\subw^2 := \F_\subw \cap \D_\subw^2$,
\item $\F_\subs := \{f \in \D_\subs :\, \E f =1,\, \exists\, c>0 \text{ s.t. } f>c\}$.
\end{compactitem}
Spaces of divergence-free vector fields: Let $y \in \R^d$,
\begin{compactitem}[]
\item $\Phi_y^\supw := \{\phi \in (L^2)^d :\, \E[ (\nabla w )\phi ] = 0\, \forall w \in \D_\subs,\, \E\phi = y\}$,
\item $\Phi_y^\sups := \{\phi \in (\D_\subs)^d :\, \nabla \cdot \phi = 0 \, \forall \omega \in \Omega, \, \E\phi = y\}$.
\end{compactitem}
And we consider the following sets of spaces:
\begin{compactitem}[]
\item $\DDD := \{\D \subset \D_\subw :\, \D \text{ dense in } \D_\subw \text{ w.r.t. } \|\cdot\|_\nabla\}$,
\item $\FFF :=\{\F \subset \F_\subw :\, \forall f \in \F_\subw\, \exists (f_n)_n \subset \F \text{ and } c>0 \text{ s.t. } \|f_n - f\|_\nabla \to 0, \ \inf_n f_n>c\, \P\text{-a.s.}\}$,
\item $\PPPhi_y := \{\Phi_y \subset \Phi_y^\supw :\, \Phi_y \text{ dense in } \Phi_y^\supw \text{ w.r.t. } \|\cdot\|_2\}$.
\end{compactitem}

We establish denseness results whose proofs are postponed to Appendix \ref{app:derivative}:

\begin{lemma}\label{lem:spacesdense}
$\D_\subs$ is dense in $L^2$, and $\D_\subs \in \DDD$, $\F_\subs \in \FFF$ and $\Phi_y^\sups \in \PPPhi_y$ for any $y \in \R^d$.
\end{lemma}

The last statement of Lemma \ref{lem:spacesdense}, while easy to see in the compact case, is quite more involved in the non-compact framework. For the proof we rely on an argument similar to the fact from differential geometry, that exact forms are closed. The next result considers modifications of the spaces in the variational expression:

\begin{proposition}\label{prop:different_spaces}
Let $V$ be a potential. $\Gamma_V$ remains unchanged if $\Phi_y^\sups$ is replaced by any of the sets  $\Phi_y\in \PPPhi_y$. If $V$ is a potential such that $V \in L^2$, then $\Gamma_V$ remains unchanged if $\F_\subs$ and $\Phi_y^\sups$ are replaced by any of the sets $\F \in \FFF$ and $\Phi_y \in \PPPhi_y$ respectively.
\end{proposition}

\begin{proof}
By definition
\begin{align}\label{e:infinf=infinf}
\Gamma_V(y)
= 2\inf_{f \in \F_\subs} \inf_{\phi \in \Phi_y^\sups} \left[ \left( \int \frac{|\nabla f|^2}{8f} + Vf d\P \right) \left( \int \frac{|\phi|^2}{2f} d\P \right)\right]^{1/2}.
\end{align}

Let $\F \in \FFF$ and choose $f \in \F_\subw$. There is $(f_n)_n \subset \F$  and $c>0$ such that $f_n \to f$ in $\|\cdot\|_{\nabla}$, $\inf_{n \in \N } f_n > c$ and $f>c$ $\P$-a.s.. Since $\partial_i f_n \to \partial_i f$ in $L^2$, also
\begin{align*}
\E[|(\partial_i f_n)^2 - (\partial_i f)^2| ] \to 0,
\end{align*}
in fact, 
$\E[|(\partial_i f_n)^2 - (\partial_i f)^2|] 
= \E[|(\partial_i f_n - \partial_i f)(\partial_i f_n + \partial_i f)|]
\leq \|\partial_i f_n - \partial_i f\|_2 \|\partial_i f_n + \partial_i f\|_2
\leq \|\partial_i f_n - \partial_i f\|_2 C$
for some constant $C>0$. Hence, $\P$-a.s.,
\begin{align*}
\big||\nabla f_n|^2/f_n - |\nabla f|^2/f\big|
&\leq {\big(|(\nabla f_n)^2 f - (\nabla f)^2 f| + |(\nabla f)^2 f - (\nabla f)^2f_n|\big)}/{(f f_n)}\\
&\leq c^{-2}\sup\{\|f\|_\infty,\|\nabla f\|_\infty^2\} \big( \sum_i |(\partial_i f_n )^2 - (\partial_i f)^2| + |f - f_n|\big),
\end{align*}
which shows $(\nabla f_n)^2 / f_n \to (\nabla f)^2 / f$ in $L^1$. In particular, $\E[(\nabla f_n)^2 / f_n]$ $\to$ $\E[(\nabla f)^2 / f]$. $\E[|Vf-Vf_n|] \leq \|V\|_2 \|f-f_n\|_2$ which converges to $0$ if $V \in L^2$. Analogously, for any $\phi \in \Phi_y^\sups$, $\E[|\phi|^2 / f_n] \to \E[|\phi|^2 /f]$ as $n \to \infty$ since $\phi$ is bounded.
Lemma \ref{lem:spacesdense} ensures $\F_\subs \in \FFF$. \eqref{e:infinf=infinf} and the fact that $\inf_f \inf_\phi \ldots = \inf_\phi \inf_f \ldots$ thus imply robustness with respect to the choice of $\F \in \FFF$ for $\Phi_y:= \Phi_y^\sups$ fixed and for $V \in L^2$. An analogous reasoning using $\Phi_y^\sups \in \PPPhi$, see Lemma \ref{lem:spacesdense}, leads to independence of the choice of $\Phi_y \in \PPPhi_y$.
\end{proof}

\section{Proof of the Variational Formula}\label{sec:proofofvarform}

In this section we are going to prove Theorem \ref{t:varform}. We follow closely the proof developed in \cite{Schroeder88} for the periodic case.

We may restrict our consideration to $y \in S^{d-1}$. Indeed, on the one hand
\begin{align*}
\lim_{r \to \infty} -(1/r) g(0,ry,\omega) = |y| \lim_{r \to \infty} -(1/r) g(0,ry/|y|,\omega)
\end{align*}
On the other hand $\Gamma_V$ is positive homogeneous:

\begin{lemma}
Let $V$ be a potential, then for $c\geq 0$, for $y \in \R^d$,
\begin{align}\label{e:linearity}
\Gamma_V(cy) = c \Gamma_V(y).
\end{align}
\end{lemma}

\begin{proof}
$\Gamma_V(0) = 0$ by choosing $\phi\equiv 0$ in the variational expression of $\Gamma_V(0)$. For $c>0$ consider the mapping $\rho_c: \Phi_y^\sups \to \Phi_{cy}^\sups$, $\phi \mapsto c\phi$. $\rho_c$ is bijective. In particular, $c\Gamma_V(y) = \Gamma_V(cy)$ and \eqref{e:linearity} follows.
\end{proof}

\subsection{Upper Bound}\label{subsec:upperbound}

Throughout this subsection we assume $(\Omega,\FF,\P,\tau)$ to be an ergodic dynamical system and $V$ to be a potential which satisfies (B) and (G). We set $\F:=\F_\subs$ and $\Phi_y:= \Phi_y^\sups$. In this subsection we prove:

\begin{proposition}\label{p:moregeneralupperbound}
For $y \in \R^d \setminus \{0\}$ $\P$-a.s.,
\begin{align*}
\limsup_{r \to \infty} - (1/r)\ln g(0,r y,\omega) \leq \Gamma_V(y).
\end{align*}
\end{proposition}

We start with

\begin{lemma}\label{lem:g_geq}
Let $0<\epsilon<1$. For $y \in \R^d \setminus \{0\}$, for $\omega \in \Omega$, 
\begin{align*}
g(0,y,\omega) \geq c_d (\epsilon |y|)^{-d} e^{-\epsilon \sqrt{2 \vmax} |y|} \int_{B(y,\epsilon|y|)} g(0,z,\omega)dz,
\end{align*}
where $c_d$ is a constant only depending on dimension $d$.
\end{lemma}

\begin{proof}
Choose $0<r<\epsilon|y|$. For $f \in C^2(\R^d)$ bounded and with bounded derivatives up to order two, under $P_y$,
\begin{align}\label{eq:martingale}
e^{-\int_0^t V_\omega(Z_s) ds} f(Z_t)
- \int_0^t (\tfrac{1}{2} \Delta f - V_\omega f)(Z_s) e^{- \int_0^s V_\omega(Z_u) du} ds, \ t \geq 0,
\end{align}
is a martingale with respect to $(\FF_t)_{t \geq 0}$, where $\FF_t:=\sigma(Z_s, 0 \leq s \leq t)$, see \cite[Theorem 2.4.2(ii)]{Pinsky1995}. By (G) one has $g(0,\cdot,\omega) \in C^2(\R^d\setminus\{0\})$ and we may choose in \eqref{eq:martingale} $f \in C^2(\R^d)$ such that $f=g(0,\cdot,\omega)$ on a neighbourhood of $B(y,r)$ which does not contain the origin. Let $\tau$ be the first exit time of $B(y,r)$. $\tau$ is a stopping time for $(\FF_t)_{t\geq 0}$, see \cite[Problem 1.2.7]{Karatzas1991}. Considering the stopped martingale, see \cite[Problem 1.3.24]{Karatzas1991}, (G) implies that
\begin{align*}
e^{-\int_0^{t \wedge \tau} V_\omega(Z_s) ds} g(0,Z_{t\wedge\tau},\omega), \ t \geq 0,
\end{align*}
is a martingale under $P_y$ with respect to $(\FF_t)_{t\geq0}$. Now, using (B) the proof can be completed as in \cite[Lemma 2.1]{Schroeder88}.
\end{proof}

Lemma \ref{lem:g_geq} implies for $y \in S^{d-1}$, for $\omega \in \Omega$,
\begin{align*}
\limsup_{r \to \infty} - \frac{1}{r} \ln g(0,ry,\omega) 
& \leq \epsilon \sqrt{2\vmax}  + \limsup_{r \to \infty} - \frac{1}{r} \ln \int_{B(ry,\epsilon r )} g(0,z,\omega) dz.
\end{align*}
This holds for arbitrary $0<\epsilon<1$; therefore, in order to prove Proposition \ref{p:moregeneralupperbound} it is sufficient to estimate
\begin{align*}
\limsup_{r \to \infty} - \frac{1}{r} \ln \int_{B(ry,\epsilon r)} g(0,z,\omega)dz.
\end{align*}

For the following choose $y \in S^{d-1}$ and let $r,\epsilon >0$.

\begin{lemma}\label{lem:feynman_kac_appl}
For $\omega \in \Omega$, $t\geq 0$,
\begin{align}\label{eq:feynman_kac_appl}
\int_{B(ry,\epsilon r)} g(0,z,\omega) dz
\geq E\left[\exp\left\{-\int_0^t V_\omega(Z_s)ds \right\} \int_0^t 1_{B(ry,\epsilon r)}(Z_s) ds \right].
\end{align}
\end{lemma}

\begin{proof}
This is a consequence of the definition of the Green function as the density for the Green measure, see \eqref{e:def_green}: Since $V \geq 0$,
\begin{align*}
\int_{B(ry,\epsilon r)} g(0,z,\omega)dz 
&= \int_0^\infty E\left[\exp\left\{-\int_0^s V_\omega(Z_u)du\right\} 1_{B(ry,\epsilon r)}(Z_s)\right]ds\\
& \geq E\left[\exp\left\{-\int_0^t V_\omega(Z_u)du\right\} \int_0^t 1_{B(ry,\epsilon r)}(Z_s)ds\right].\qedhere
\end{align*}
\end{proof}

Let $f \in \F$, $\phi \in \Phi_y$ and introduce for $a>0$ the drift
\begin{align*}
b:=\frac{\nabla f}{2f} + a \frac{\phi}{f}.
\end{align*}
Since $\phi$, $\nabla f$ are bounded and since $f$ is strictly bounded away from zero, $b$ is bounded. Consequently, for $\omega \in \Omega$ $b_\omega$ is bounded on $\R^d$ and there exists a solution $P_0^{b_\omega}$ to the martingale problem $\MP(b_\omega,\id)$ starting at $\delta_0$.
Introduce $\beta_t:=Z_t - \int_0^t b_\omega(Z_s) ds$. By Cameron-Martin-Girsanov formula $(\beta_t)_t$ is a Brownian motion under $P_0^{b_\omega}$, see \cite[Theorem 3.5.1, Corollary 3.5.13, proof of Proposition 5.3.6]{Karatzas1991}, and for any $t \geq 0$,
\begin{align*}
\frac{dP_0[(Z_s)_{s \leq t} \in \cdot]}{dP_0^{b_\omega}[(Z_s)_{s \leq t} \in \cdot]}
& =\exp\left\{-\int_0^t b_\omega(Z_s) d\beta_s - \frac{1}{2}\int_0^t |b_\omega|^2(Z_s)ds\right\},
\end{align*}
where we used the distributive law for stochastic integration, see \cite[(2.8.4)]{Durrett96}. 
This and \eqref{eq:feynman_kac_appl} lead to

\begin{lemma}\label{lem:cameronmartin}
For all $\omega \in \Omega$, $t\geq 0$,
\begin{align*}
\int_{B(ry,\epsilon r)} g(0,z,\omega) dz
\geq E^{b_\omega}_0\bigg[\exp\bigg\{& -\frac{1}{2} \int_0^t|b_\omega|^2(Z_s)ds - \int_0^t b_\omega(Z_s) d\beta_s\\
&- \int_0^t V_\omega(Z_s)ds\bigg\} \int_0^t 1_{B(ry,\epsilon r)}(Z_s)ds\bigg].
\end{align*}
\end{lemma}

For $\omega \in \Omega$, $t>0$, for $\delta>0$ with $\epsilon > \delta$ introduce the event $\AAa_\omega(t,\epsilon,\delta):= A_{\omega,t}^1 \cap A_{\omega,t}^2 \cap A_{\omega,t}^3$ where
\begin{align*}
A_{\omega,t}^1 := \bigg\{ \bigg|\frac{1}{t} \int_0^t |b_\omega|^2(Z_s)ds - & \E[|b|^2 f]\bigg| < \delta \bigg\},\
A_{\omega,t}^2 := \bigg\{\bigg|\frac{1}{t} \int_0^t V_\omega(Z_s)ds - \E[V f]\bigg| < \delta \bigg\},\\
A_{\omega,t}^3 & := \bigg\{\bigg|\frac{Z_s}{s} - ay\bigg|<a\delta \, \forall s > t-(\epsilon-\delta)t\bigg\}.
\end{align*}
Due to the ergodic properties of the underlying dynamical system we have

\begin{lemma}\label{l:gammato1}
There exists $\GG(a,f,\phi) \in \FF$, $\P[\GG(a,f,\phi)]=1$, such that for $\omega \in \GG(a,f,\phi)$, for $0<\delta<\epsilon$,
\begin{align*}
P_0^{b_\omega}[\AAa_\omega(t,\epsilon,\delta)] \to 1 \text{ as } t \to \infty.
\end{align*}
\end{lemma}

\begin{proof}
The Ergodic Theorem given in \cite[(A.9)]{DelTenno2009SpecialExamples} shows that for $i=1,2$ $\P$-a.s.\ $P_0^{b_\omega}$-a.s.\ $\lim_{t \to \infty}1_{A_{\omega,t}^{i}} = 1$. For the examination of $A_{\omega,t}^3$ recognise that by the definition of $b$ and since $Z_s = \beta_s + \int_0^sb_\omega(Z_u)du$,
\begin{align*}
\frac{Z_s}{s} 
& = \frac{\beta_s}{s} + \frac{1}{s}\int_0^s \frac{\nabla f_\omega}{2f_\omega}(Z_u)du + \frac{1}{s}\int_0^s\frac{a\phi_\omega}{f_\omega}(Z_u)du.
\end{align*}
$(\beta_t)_t$ is a Brownian motion under $P_0^{b_\omega}$, thus, the first term in the last expression converges to zero $P_0^{b_\omega}$-a.s.. By the Ergodic Theorem and integration by parts \eqref{eq:partialintegration} $\P$-a.s.\ $P_0^{b_\omega}$-a.s.\
\begin{align*}
&\lim_{s \to \infty} \frac{1}{s}\int_0^s \frac{\nabla f_\omega}{2f_\omega}(Z_u)du = \frac{1}{2}\E[\nabla f]=0,
&\lim_{s \to \infty} \frac{1}{s}\int_0^s \frac{a\phi_\omega}{f_\omega}(Z_u)du = a\E[\phi] = ay.
\end{align*}
It follows $\P$-a.s.\ $P_0^{b_\omega}$-a.s.\ $\lim_{t \to \infty} 1_{A_{\omega,t}^3} = 1$.

By dominated convergence we get $\P$-a.e.\ $P_0^{b_\omega}[A_{\omega,t}^i] \to 1$, $i=1,2,3$. Here the $\P$-a.s.\ convergences guaranteed by the Ergodic Theorem depend on the choice of $b$ and on the functions over which the space and time averages are taken.
\end{proof}

We continue by choosing $t:=r/a$. On $\AAa_\omega(r/a,\epsilon,\delta)$ one has for $(1-(\epsilon-\delta))r/a <s<r/a$ that $|Z_s - ry| 
\leq |Z_s-asy| + |a(\frac{r}{a}-s)y| 
< \epsilon r$. Hence, on $\AAa_\omega(r/a,\epsilon,\delta)$,
\begin{align*}
\int_0^{r/a} 1_{B(ry,\epsilon r)}(Z_s)ds
\geq \int_{(1-(\epsilon-\delta))r/a}^{r/a} 1_{B(ry,\epsilon r)}(Z_s)ds
\geq (\epsilon-\delta)\frac{r}{a}.
\end{align*}
We deduce with Lemma \ref{lem:cameronmartin} (choose $t = r/a$)
\begin{align}\label{eq:zwischenergebnis}
\int_{B(ry,\epsilon r)} g(0,z,\omega) dz
&\geq\exp\bigg\{-\frac{r}{a}\bigg(2\delta+\frac{1}{2}\E[|b|^2 f]+\E[Vf]\bigg)\bigg\}\\\nonumber
&\qquad\cdot(\epsilon-\delta)\frac{r}{a} E_0^{b_\omega}\bigg[\exp\bigg\{-\int_0^{r/a} b_\omega(Z_s)d\beta_s\bigg\},\AAa_\omega(r/a,\epsilon,\delta)\bigg].
\end{align}

We estimate the latter:

\begin{lemma}\label{eq:zwischenergebnis2}
Let $\gamma_\omega := P_0^{b_\omega}[\AAa_\omega(r/a,\epsilon,\delta)]$. For all $\omega \in \Omega$,
\begin{align*}
&E_0^{b_\omega}\bigg[\exp\bigg\{-\int_0^{r/a} b_\omega(Z_s) d\beta_s\bigg\},\AAa_\omega(r/a,\epsilon,\delta)\bigg]
\geq \gamma_\omega \exp\{-\gamma_\omega^{-1/2} \|b\|_\infty (r/a)^{1/2}\}.
\end{align*}
\end{lemma}

\begin{proof}
Introduce for $t \geq 0$ the process $Y_t:=\int_0^t b_\omega(Z_s) d\beta_s$. Jensen's inequality gives
\begin{align}\label{eq:jensen}
&E_0^{b_\omega}\bigg[\exp\bigg\{-\int_0^{r/a} b_\omega(Z_s) d\beta_s\bigg\},\AAa_\omega\bigg]
\geq \gamma_\omega \exp\{-\gamma_\omega^{-1}E_0^{b_\omega}[|Y_{r/a}|,\AAa_\omega]\}.
\end{align}
By Hölder's inequality and It\^o isometry,
\begin{align}\label{eq:hoelder}
E_0^{b_\omega}[|Y_{r/a}|,\AAa_\omega]
\leq \gamma_\omega^{1/2} E_0^{b_\omega}[(Y_{r/a})^2]^{1/2}
= \gamma_\omega^{1/2} E_0^{b_\omega}\bigg[\int_0^{r/a}|b_\omega|^2(Z_s)ds\bigg]^{1/2}.
\end{align}
Since $b_\omega$ is bounded we have $E_0^{b_\omega}[\int_0^{r/a}|b_\omega|^2(Z_s)ds] \leq (r/a) \|b\|_\infty^2$. This together with \eqref{eq:jensen} and \eqref{eq:hoelder} shows the statement.
\end{proof}

The choice of $\epsilon>\delta>0$ was arbitrary, hence, Lemma \ref{l:gammato1}, Estimate \eqref{eq:zwischenergebnis} and Lemma \ref{eq:zwischenergebnis2} imply on $\GG(a,f,\phi)$,
\begin{align}\label{eq:finalprof}
\limsup_{r \to \infty} - \frac{1}{r}\ln \int_{B(ry,\epsilon r)} g(0,z,\omega) dz
&\leq a^{-1}(2^{-1}\E[b^2f] + \E[Vf]).
\end{align}
The definition of $b$ shows
\begin{align*}
\E[b^2f] = \E\left[\frac{|\nabla f|^2}{4f}\right] + 2a\left\langle \frac{\nabla f}{2f} , \phi \right\rangle + a^2 \E\left[\frac{\phi^2}{f}\right].
\end{align*}
Using integration by parts \eqref{eq:partialintegration} the middle term equals
$a \langle (\nabla f)/f , \phi \rangle 
= a \langle \nabla\ln f , \phi \rangle 
= - a \langle \ln f ,\nabla \phi \rangle
= 0$
since $\phi \in \Phi_y$.

Denote the right-hand side of \eqref{eq:finalprof} by $R(a,f,\phi)$. We want to minimise over $a>0$, $f \in \F$ and $\phi \in \Phi_y$. The exceptional sets $\GG(a,f,\phi)^\complemetc$ on which \eqref{eq:finalprof} does not hold necessarily depend on $a$, $f$ and $\phi$. In order to be sure that these do not add up to a nontrivial set, note that there exist `minimising sequences' $(f_n)_n\subset\F$, $(a_{n})_{n}\subset(0,\infty)$ and $(\phi_{n})_{n}\subset\Phi_y$ such that
\[
\inf_{f \in \F}\inf_{a>0}\inf_{\phi \in \Phi_y}R(a,f,\phi) = \inf_{n \in \N}R(a_n,f_n,\phi_n).
\]
The considered families are countable, thus the union of the exceptional sets $\GG(a_n,f_n,\phi_n)^\complemetc$ has measure zero. In the remaining part of this subsection $\P$-a.s.\ expressions refer to $\bigcap_n \GG(a_n,f_n,\phi_n)$.

We get from \eqref{eq:finalprof} $\P$-a.s.\ that 
$\limsup_{r \to \infty} - (1/r)\ln\int_{B(ry,\epsilon r)} g(0,z,\omega) dz$
is less than or equal to
\begin{align*}
\inf_{f \in \F} \inf_{a > 0} \left[ \frac{1}{a} \left( \int \frac{|\nabla f|^2}{8f} + Vfd\P \right) + a \left( \inf_{\phi \in \Phi_y} \int \frac{\phi^2}{2f} d\P \right)\right].
\end{align*}
Set $v:=\E [|\nabla f|^2/(8f) + Vf]$ and $w:=\inf_{\phi \in \Phi_y} \E [\phi^2/(2f)]$. $v \geq \E[V] \min_\Omega f>0$ and $2w \geq \E[\phi^2]/\max_\Omega f\geq |y|^2/\max_\Omega f > 0$ by Jensen's inequality. The infimum of $a \mapsto v/a + aw$ for positive $a$ is therefore achieved at $a_{\mathrm{min}}=\sqrt{v/w}$ with minimum $2\sqrt{vw}$. Thus, $\P$-a.s.,
\begin{align*}
\limsup_{r \to \infty} - \frac{1}{r}\ln \int_{B(ry,\epsilon r)} g(0,z,\omega) dz
\leq \inf_{f \in \F} 2 \left[ \left( \int \frac{|\nabla f|^2}{8f} + Vf d\P \right) \left(\inf_{\phi \in \Phi_y} \int \frac{\phi^2}{2f} d\P \right) \right]^{1/2}.
\end{align*}
Therefore, for any $y \in \R^d \setminus\{0\}$ $\P$-a.s.\ the upper bound holds which shows the statement of Proposition \ref{p:moregeneralupperbound}.

\subsection{Lower Bound}\label{subsec:lowerbound}

In this subsection we are going to show that for potentials $V$ subject to conditions (G) and (E1) $\P$-a.s.\ for any $y \in \R^d \setminus \{0\}$,
\begin{align}\label{eq:upperbound}
\liminf_{r \to \infty} - \frac{1}{r}\ln g(0,ry,\omega) 
\geq \RR(\sigma)(y).
\end{align}

Note that (E1) contains some kind of ergodicity condition on $\P$, and we do not need to require explicitly $(\Omega,\FF,\P,\tau)$ to be an ergodic dynamical system . We start with

\begin{proposition}\label{prop:alpha_geq_R_Lambda}
Let $V$ be a potential satisfying (G). Then for $\omega \in \Omega$, for $y \in \R^d\setminus\{0\}$,
\begin{align*}
\liminf_{r \to \infty} - \frac{1}{r}\ln g(0,ry,\omega) 
\geq \RR(-\Lambda_\omega)(y).
\end{align*}
\end{proposition}

\begin{proof}
Since $V_\omega \geq 0$ as well as $g(0,\cdot,\omega) \geq 0$ we have $V_\omega g(0,\cdot,\omega)\geq 0$ on $\R^d\setminus\{0\}$. Thus (G) gives $\Laplace g(0,\cdot,\omega) \geq 0$ on $\R^d\setminus\{0\}$ which implies that $g(0,\cdot,\omega)$ is subharmonic on $\R^d\setminus\{0\}$, see \cite[Paragraph 1.II.8]{Doob1984}. Hence, for all $x \in \R^d$ such that $|x|>1$,
\begin{align*}
g(0,x,\omega) \leq c(d) \int_{B(x,1)} g(0,z,\omega) dz,
\end{align*}
where $c(d)$ is a constant depending only on $d$. Consequently, it suffices to get the lower bound for
\begin{align*}
\liminf_{r \to \infty} - \frac{1}{r} \ln \int_{B(ry,1)} g(0,z,\omega) dz.
\end{align*}

Information on the exponential decay of $r \mapsto g(0,ry,\omega)$ can be obtained by determining those $\lambda \in \R^d$ for which $\int_{\R^d} e^{\lambda z} g(0,z,\omega)dz$ is finite. We calculate
\begin{align}\label{e:2:first}
\begin{split}
\int_{\R^d} e^{\lambda z} g(0,z,\omega) dz
&= \int_0^\infty  E_0\bigg[\exp\bigg\{\lambda Z_t-\int_0^t V_\omega (Z_s) ds\bigg\} \bigg] dt.
\end{split}
\end{align}
Considering Brownian motion with constant drift $\lambda$ the latter equals
\begin{align*}
\int_0^\infty  e^{t\lambda^2/2}E_0^\lambda\bigg[ \exp\bigg\{-\int_0^t V_\omega (Z_s) ds\bigg\} \bigg] dt.
\end{align*}
Exponential decay of $E_0^\lambda[e^{-\int_0^t V_\omega (Z_s) ds} ]$ is the quenched free energy and we get for those $\lambda$ which satisfy $\lambda^2/2 < -\Lambda_\omega(\lambda)$,
\begin{align*}
\int_0^\infty  e^{t\lambda^2/2}E_0^\lambda\bigg[ \exp\bigg\{-\int_0^t V_\omega (Z_s) ds \bigg\} \bigg] dt
\leq c < \infty
\end{align*}
for some $c=c(\omega,\lambda)>0$.
With \eqref{e:2:first} we conclude for any $r \geq 0$,
\begin{align*}
\int_{B(ry,1)}e^{\lambda z}g(0,z,\omega)dz \leq \int_{\R^d}e^{\lambda z}g(0,z,\omega)dz \leq c.
\end{align*}
Let $\bar z$ be the point in $B(ry,1)$ for which $e^{\lambda z}$ is minimal on $B(ry,1)$. Since $|\lambda(\bar z - ry)| \leq |\lambda|$,
\begin{align*}
\int_{B(ry,1)}g(0,z,\omega)dz 
= e^{- \lambda \bar z} \int_{B(ry,1)}e^{\lambda \bar z}  g(0,z,\omega)dz 
\leq c e^{- \lambda \bar z}
\leq c e^{|\lambda|} e^{- r \lambda y}.
\end{align*}
Thus, for $\lambda$ with $\lambda^2/2<-\Lambda_\omega(\lambda)$ we have $\liminf_{r \to \infty} - \frac{1}{r}\ln g(0,ry,\omega) 
\geq \langle y, \lambda \rangle$.
\end{proof}

Proposition \ref{prop:alpha_geq_R_Lambda} together with the following result implies \eqref{eq:upperbound}.
\begin{lemma}
Let $V$ be a potential satisfying (E1), then $\P$-a.s.\ for any $y \in \R^d$,
\begin{align}\label{eq:R_Lambda_geq_R_sigma}
\RR(-\Lambda_\omega)(y) \geq \RR(\sigma)(y).
\end{align}
\end{lemma}

\begin{proof}
Choose a dense countable subset $\{\lambda_n: \, n \in \N\}$ of the set $\{\lambda \in \R^d: \, \lambda^2/2 < \sigma(\lambda)\}$. Since the scalar product is continuous, one has $\sup\{ \langle y,  \lambda_n \rangle: \, \lambda_n^2/2<\sigma(\lambda_n), \, n\in\N\}= \RR(\sigma)$. Let $E_n \in \FF$ denote the exceptional set for $\lambda_n$ in condition (E1). \eqref{eq:R_Lambda_geq_R_sigma} is valid on $(\bigcup_nE_n)^\complemetc$ which is a set of probability one. Moreover, $\bigcup_nE_n$ does not depend on $y$.
\end{proof}

\subsection{Equality of the Upper and Lower Estimates}\label{subsec:equivalence}

The proof of Theorem \ref{t:varform} is finished as soon as we establish equality between upper and lower estimate, that is we have to show under sufficient conditions
\begin{align}\label{eq:toshow_for_equality}
\RR(\sigma) = \Gamma_V.
\end{align}

We introduce the functional $I$ appearing often in the context of large deviations of Markov processes as defined in \cite[(1.12)]{Donsker75b}: For $f \in \F_\subw$ and $\lambda \in \R^d$ set
\begin{align*}
I(f) := - \inf_{u \in \U} \int \frac{L^\lambda u}{u} fd\P.
\end{align*}
The following characterisation of the functional $I$ is the same as the one given in \cite[Lemma 3.3]{Donsker1976PrincipalEigenvalue} on $\R^d$ and is proven analogously:

\begin{proposition}\label{prop:formulaforI}
Let $f \in \F_\subw^2$ and $\lambda \in \R^d$. For any $\D \in \DDD$,
\begin{align}\label{eq:formulaforI}
I(f)
=  \E \left[ \frac{|\nabla f|^2}{8f}\right] + \frac{\lambda^2}{2} - \frac{1}{2} \inf_{w \in \D} \E[ |\lambda - \nabla w|^2f ].
\end{align}
In particular, for potentials $V$,
\begin{align}\label{eq:formula_for_sigma}
\sigma(\lambda) 
& = \inf_{f \in \F_\subw^2} \left\{ \E\left[ Vf +  \frac{|\nabla f|^2}{8f} \right] + \frac{|\lambda|^2}{2} - \frac{1}{2} \inf_{w \in \D} \E[ |\lambda - \nabla w |^2f ]\right\}.
\end{align}
\end{proposition}

\begin{proof}
The right-hand side of \eqref{eq:formulaforI} is independent of the choice of $\D \in \DDD$ as outlined in Lemma \ref{lem:propertiesofHandsigma} below. Without restriction assume $\D= \D_\subw^2$. We start calculating
\begin{align*}
I(f)
& = - \inf_{w \in \D} \int \frac{L^\lambda e^w}{e^w} f d \P
= - \inf_{w \in \D} \int \left(L^\lambda w + \frac{1}{2} |\nabla w|^2 \right)f d\P.
\end{align*}
For $w \in \D$ define $h := (1/2) \ln f -w$. Hence, $w = (1/2)\ln f - h$ and a straightforward calculation shows
\begin{align*}
L^\lambda w + \frac{1}{2} |\nabla w|^2 
& = \frac{\Laplace f}{4f} - \frac{|\nabla f|^2}{8 f^2} + \frac{\lambda \nabla f}{2f} 
- \frac{\Laplace h}{2} + \frac{|\nabla h|^2}{2} -\lambda \nabla h
- \frac{\nabla f \nabla h}{2f}.
\end{align*}
$h$ is in $\D$, moreover, the mapping $w \mapsto h$ is bijective from $\D$ to $\D$. Therefore,
\begin{align*}
I(f)
= & - \int \left(\frac{\Laplace f}{4f} - \frac{|\nabla f|^2}{8 f^2} + \frac{\lambda \nabla f}{2f}\right)f d\P \\
& \quad - \inf_{w \in \D} \int \left(- \frac{\Laplace w}{2} + \frac{|\nabla w|^2}{2} -\lambda \nabla w - \frac{\nabla f \nabla w}{2f} \right)f d\P.
\end{align*}
Integration by parts \eqref{eq:partialintegration} implies
$- \E [ (\Laplace w) f ] = \E[ \nabla w \nabla f]$,
$\E [\Laplace f] = 0$, and
$\E[ \lambda \nabla f] = 0$. Thus,
\begin{align*}
I(f) = \int \frac{|\nabla f|^2}{8f}d\P - \inf_{w \in \D} \int \left(\frac{|\nabla w|^2}{2} - \lambda \nabla w \right)f d\P
\end{align*}
which shows the statement.
\end{proof}

For $f \in L^1$ bounded such that $f\geq c$ for some $c>0$ and $\E f = 1$ we introduce an inner product on $(L^2)^d$: Let for $\phi$, $\psi \in (L^2)^d$,
\begin{align}\label{eq:def_scalarproduct_wrt_f}
\langle \phi,\psi \rangle_f := \E[\phi \psi f]
\end{align}
and set $\|\psi\|_f:= \E[\psi^2 f]^{1/2}$.
We define for $b \in (L^2)^d$, $f \in \F_\subw$, $\D \in \DDD$,
\begin{align}\label{e:DefH}
H(b,f):= \inf_{w \in \D} \E[|b - \nabla w|^2 f],
\end{align}
and
\begin{align} \label{eq:defK}
K(f) &:= \E\left[ \frac{|\nabla f|^2}{8f} +  Vf\right].
\end{align}

We collect some properties of $H$ and $\sigma$:

\begin{lemma}\label{lem:propertiesofHandsigma}
The definition of $H$ is independent of the choice of $\D \in \DDD$.
$H(\cdot,f)^{1/2}$ is a seminorm on $((L^2)^d,\|\cdot\|_f)$ factorised by the vector subspace consisting of all $\nabla w$ with $w \in \D_\subs$. In particular, for $b_1$, $b_2 \in (L^2)^d$,
\begin{align}\label{e:triangle}
H(b_1+b_2,f)^{1/2} \leq H(b_1,f)^{1/2} + H(b_2,f)^{1/2}.
\end{align}
$H(\cdot,f)$ is continuous with respect to $\|\cdot\|_2$. Let $\eta \in S^{d-1}$, then
\begin{align}\label{e:bounds_on_HK}
0< \ess \inf_\P f & \leq H(\eta,f) \leq 1 \text{ and } H(\eta, 1)=1.
\end{align}
Let $V$ be a potential. For any $\lambda \in \R^d$,
\begin{align}\label{eq:sigma_invariant_under_reflections}
\sigma(\lambda) = \sigma(-\lambda).
\end{align}
The mapping from $\R^d \to \R$,
\begin{align}\label{eq:M_concave}
\lambda \mapsto \sigma(\lambda)-\lambda^2/2
= \inf_{f \in \F_\subw^2} \bigg\{ K(f) - \frac{1}{2} H(\lambda,f)\bigg\}
\end{align}
is concave. Moreover, for $\eta \in S^{d-1}$, the mapping from $[0,\infty) \to \R$,
\begin{align}\label{eq:M_decreasing}
s \mapsto \sigma(s \eta)-s^2/2
\end{align}
is monotone decreasing. In particular,
\begin{align}\label{eq:Maximum_of_M}
\sup\{\sigma(\lambda)-\lambda^2/2: \ \lambda \in \R^d\} = \sigma(0).
\end{align}
One has
\begin{align}\label{eq:sigma_independent_ofchoiceof_F}
\sigma(0) = \inf_{f \in \F_\subw^2} K(f) \text{, and if $V \in L^2$, for all $\F \in \FFF$ we get } \sigma(0) = \inf_{f \in \F} K(f).
\end{align}
\end{lemma}

\begin{proof}
Independence of the choice of $\D \in \DDD$: Fix $f \in \F_\subw$ and $b \in (L^2)^d$. For $w \in \D_\subw$ let $w_n$ be a sequence in $\D$ converging to $w$ in $\|\cdot\|_{\nabla}$. Then $\nabla w_n \to \nabla w$ in $L^2$. By boundedness of $f$, also $\nabla w_n \to \nabla w$ in $\|\cdot\|_f$. This shows $\E[|b - \nabla w_n|^2 f] \to \E[ |b - \nabla w|^2 f]$.

One has $H(b,f)^{1/2}=\inf_{w \in \D_\subs} \| b-\nabla w\|_f$. Hence, $H(\cdot,f)^{1/2}$ is a seminorm on $((L^2)^d,\|\cdot\|_f)$ factorised by the vector space consisting of all $\nabla w$ with $w \in \D_\subs$, see e.g.\ \cite[22.13.b]{Schechter1997}. $H(\cdot,f)$ is continuous with respect to $\|\cdot\|_2$, since
\begin{align}\label{eq:Hcontinuous}
H(b,f)^{1/2} \leq \|f\|_\infty^{1/2} H(b,1)^{1/2} \leq \|f\|_\infty^{1/2} \|b\|_2.
\end{align}

For \eqref{e:bounds_on_HK} choose $w\equiv 0$ which shows $\E[|\eta- \nabla w|^2f] = \int |\eta|^2 f d\P = 1$, thus $H(\eta,f) \leq 1$. On the other hand, integration by parts gives $\langle \eta, \nabla w \rangle=0$. Hence, $\eta, \nabla w$ are orthogonal to each other in $L^2$ which leads to
\begin{align*}
H(\eta, 1) = \inf_{w \in \D} \E[ |\eta - \nabla w|^2] = \inf_{w \in \D}[1 + \|\nabla w \|_2^2] = 1.
\end{align*}
Therefore, $H(\eta, f) \geq (\ess\inf_\P f) H(\eta,1) = \ess\inf_\P f>0$.

Invariance of $\sigma$ under reflections follows from the fact, that $w \mapsto -w$ is a bijective mapping on $\D_\subs$.

The fact that $H(\cdot,f)^{1/2}$ is a seminorm and convexity of $x \mapsto x^2$ show that $\lambda \mapsto K(f) - (1/2) H(\lambda,f)$ is concave. The infimum over concave functions is again concave which shows concavity of \eqref{eq:M_concave}.

One has for $\lambda \in \R^d$,
\begin{align}\label{eq:H(lambda)=lambda^2H}
H(\lambda,f)=|\lambda|^2 H(\eta,f),
\end{align}
where $\eta\in S^{d-1}$ is in direction of $\lambda$. This follows for $\lambda =0$ by choosing $w \equiv 0$ in the definition of $H(0,f)$. For $\lambda \neq 0$ use the fact, that the mapping $w \mapsto |\lambda|w$ is bijective on $\D$. This shows \eqref{eq:M_decreasing}.

\eqref{eq:Maximum_of_M} follows from \eqref{eq:sigma_invariant_under_reflections} and \eqref{eq:M_decreasing}.

Formula \eqref{eq:sigma_independent_ofchoiceof_F} is a consequence of the representation \eqref{eq:formula_for_sigma}. Independence of the choice of $\F \in \FFF$ follows as in Proposition \ref{prop:different_spaces}.
\end{proof}

The following is a consequence of the representation for $\sigma$ obtained in \eqref{eq:formula_for_sigma}:

\begin{proposition}\label{prop:first_equality}
Let $V$ be a potential. If $\sigma(0) > 0$, then for any $y \in \R^d$,
\begin{align*}
\RR(\sigma)(y)
=\sup_{\eta \in S^{d-1}} \inf_{f \in \F_\subw^2} 
\left[ 2 K(f) \frac{\langle y,\eta\rangle^2}{H(\eta, f)} \right]^{1/2}.
\end{align*}
Let $y \neq 0$. Then $\RR(\sigma)(y) > 0$ if and only if $\sigma(0) > 0$. $\RR(\sigma)\equiv -\infty$ if and only if $\sigma(0)=0$.
\end{proposition}

\begin{proof}
Let $\F:=\F_\subw^2$. For $\eta \in S^{d-1}$ denote by $S_\eta$ the set of $\lambda \in \R^d$ such that $|\lambda|^2/2 < \sigma (\lambda)$ and $\lambda$ parallel to $\eta$.
By \eqref{eq:sigma_invariant_under_reflections} $S_\eta=-S_{-\eta}$. We can rewrite
\begin{align}\label{eq:R=supsup}
\RR(\sigma)(y)
=\sup_{\eta \in S^{d-1}}\sup_{\lambda \in S_\eta} \langle y, \eta \rangle  |\lambda|
=\sup_{\eta \in S^{d-1}}\langle y, \eta \rangle \sup_{\lambda \in S_\eta} |\lambda|,
\end{align}
where the second equality is valid if there exists $\eta$ such that $S_\eta \neq \emptyset$. This is the case if $\sigma(0)>0$. The reverse is also true:

Let $\eta \in S^{d-1}$. By \eqref{eq:formula_for_sigma}, \eqref{eq:H(lambda)=lambda^2H} for $\lambda$ parallel to $\eta$ one has
$|\lambda|^2/2<\sigma(\lambda)$ if and only if
\begin{align*}
0 < \inf_{f \in \F} \bigg\{K(f) - \frac{|\lambda|^2}{2}H(\eta,f)\bigg\}=:M(|\lambda|).
\end{align*}
As outlined in \eqref{eq:M_concave} and \eqref{eq:M_decreasing} the latter is concave decreasing in $|\lambda|$. Hence,
$S_\eta \neq \emptyset$ if and only if $M(0)=\inf_{f \in \F}K(f)=\sigma(0)>0$.

Assume $\inf_{f \in \F}K(f)>0$. Since $M$ is concave and continuous, $\bar S_\eta$ equals the set of $\lambda$ parallel to $\eta$ for which
\begin{align*}
0 \leq \inf_{f \in \F} \bigg\{K(f) - \frac{|\lambda|^2}{2}H(\eta,f)\bigg\}.
\end{align*}
This is true if and only if for any $f \in \F$,
\begin{align*}
0 \leq K(f) - \frac{|\lambda|^2}{2}H(\eta,f).
\end{align*}
As $H(\eta,f)>0$, see \eqref{e:bounds_on_HK}, the latter is equivalent to
$|\lambda| \leq \left( 2K(f)/H(\eta,f) \right)^{1/2}$.
We get in case of $\sigma(0)>0$
\begin{align*}
\sup_{\lambda \in S_\eta} |\lambda|
=\sup_{\lambda \in \bar S_\eta} |\lambda|
=\inf_{f \in \F} \left( 2K(f)/H(\eta,f) \right)^{1/2}.
\end{align*}
This together with \eqref{eq:R=supsup} shows the statement.
\end{proof}

On the other hand we have		

\begin{proposition}\label{t:alternativerep1}
Let $V$ be a potential. For $y \in \R^d$, for $\D \in \DDD$,
\begin{align*}
\Gamma_V(y)
& = \inf_{f \in \F_\subs} \sup_{\eta \in S^{d-1}}
\left[ 2 K(f) \frac{\langle y,\eta\rangle^2}{H(\eta, f)} \right]^{1/2}.
\end{align*}
If additionally $V \in L^2$, then the set $\F_\subs$ can be replaced by any $\F \in \FFF$.
\end{proposition}

The first statement is true by the following lemma which is a consequence of orthogonal projection in Hilbert spaces. Independence of the choice of $\F$ follows from Proposition \ref{prop:different_spaces}.

\begin{lemma}\label{lem:projectionlemma}
Let $f \in \F_\subw$, $y \in \R^d$, $\Phi_y \in \PPPhi_y$ and $\D \in \DDD$. Then
\begin{align*}
\inf_{\phi \in \Phi_y} \E \left[\frac{|\phi|^2}{f} \right]
= \sup_{\eta\in S^{d-1}} \frac{\langle y, \eta \rangle^2}{H(\eta,f)}.
\end{align*}
\end{lemma}

\begin{proof}
If $y=0$ the statement is clear since $\phi \equiv 0 \in \Phi_0$. Assume $y \neq 0$. Since $H$ does not depend on the choice of $\D \in \DDD$, see Lemma \ref{lem:propertiesofHandsigma}, set $\D:= \D_\subs$. The left-hand side does not depend on the choice of $\Phi_y \in \PPPhi_y$, see Proposition \ref{prop:different_spaces}, and we can choose $\Phi_y:= \Phi_y^\supw$. Recall, that on $(L^2)^d$ we have the inner product $\langle \phi, \psi \rangle_f := \E[\phi \psi f]$ for $\phi$, $\psi \in (L^2)^d$, see \eqref{eq:def_scalarproduct_wrt_f}. We write $\phi \perp_f \psi$ for $\langle \phi , \psi \rangle_f = 0$. Further consider the subspaces of $(L^2)^d$
\begin{align*}
L & := \{ \eta - \nabla w: \, \eta \in \R^d, \, w \in \D\},\\
K & := \{ \phi \in (L^2)^d: \, \langle \nabla w, \phi \rangle_f =0 \, \forall w \in \D, \, \E[ \phi f ] =0\}.
\end{align*}
By definition, $L^{\perp_f} = K$. Therefore $\bar L \oplus K = (L^2)^d$ which can be understood as a variant of Weyl's decomposition, see \eqref{eq:weyl}.
For fixed $y \in S^{d-1}$ and $f \in \F_\subw$ consider the mapping
\begin{align*}
F: \Phi_y \to K: \phi \mapsto (y - \phi)/f.
\end{align*}
Indeed, $F$ maps to $K$: Let $\phi \in \Phi_y$, then $\E [ F (\phi) f]= y-\E[\phi]=0$ and $\langle \nabla w, F(\phi)f \rangle = - \langle \nabla w, \phi \rangle = 0$ for $w \in \D$. Moreover, $F$ is bijective with inverse $F^{-1}: \psi \mapsto y-f\psi $. For $w \in (L^2)^d$ let $\|w-K\|_f:= \inf_{\phi \in K} \|w-\phi\|_f$.

We calculate
\begin{align*}
\inf_{\phi \in \Phi_y} \E\left[\frac{\phi^2}{f}\right] 
 = \inf_{\psi \in K} \E \left[\frac{|y -f\psi|^2}{f}\right]
 = \inf_{\psi \in K} \E \bigg[ \bigg| \frac{y}{f} -\psi \bigg|^2 f\bigg]
= \bigg\| \frac{y}{f} -K\bigg\|_f^2.
\end{align*}
As a consequence of orthogonality on $(L^2(f\P))^d$,
\begin{align}\label{eq:hilbertcorr}
\bigg\|\frac{y}{f}-K\bigg\|_f = \sup_{v \in L, \, \|v\|_f=1}\bigg\langle \frac{y}{f},v\bigg\rangle_f.
\end{align}
In fact, $K$ is closed and for $w \in (L^2)^d$ there exists $\phi_0 \in K$ such that $\|w-K\|_f=\|w-\phi_0\|_f$ and $w-\phi_0 \in K^{\perp_f}$, see e.g.\ the proof of \cite[Theorem 12.4]{Rudin1991Functional}. Then,
$
\|w-K\|_f 
= \|w-\phi_0\|_f 
= \sup_{v \in \overline L, \, \|v\|_f=1}\langle w-\phi_0,v\rangle_f 
= \sup_{v \in \overline L, \, \|v\|_f=1}\langle w,v\rangle_f 
$.
Choose $(v_n)_n \subset L$ converging to $v \in \overline L$ with respect to $\|\cdot\|_f$, where $\|v\|_f=1$. Then $v_n/\|v_n\|_f \in L$ and converges to $v$. Therefore, in \eqref{eq:hilbertcorr} it is sufficient to take supremum only over $v \in L$ with $\|v\|_f=1$. We continue
\begin{align*}
\bigg\|\frac{y}{f} -K\bigg\|_f^2
& = \sup_{v \in L,\, \|v\|_f=1}\bigg\langle \frac{y}{f},v\bigg\rangle_f^2
= \sup_{\begin{subarray}{l} \eta \in \R^d,\, w \in \D: \\ \eta - \nabla  w \neq 0 \end{subarray}} \frac{\langle y/f, \eta - \nabla w \rangle_f^2}{\|\eta -\nabla w\|_f^2}\\
& \overset{(i)}{=} \sup_{\begin{subarray}{l} \eta \in \R^d \setminus\{0\}, \, w \in \D\end{subarray}}  \frac{\langle y, \eta \rangle^2}{\|\eta -\nabla w\|_f^2}
\overset{(ii)}{=} \sup_{\begin{subarray}{l} \eta \in S^{d-1},\, w \in \D\end{subarray}} \frac{\langle y, \eta \rangle^2}{\|\eta -\nabla w\|_f^2}
= \sup_{\eta \in S^{d-1}}\frac{\langle y, \eta \rangle^2}{H(\eta,f)}.
\end{align*}
For equality $(i)$ we used integration by parts and the fact that $\eta = \nabla w$ if and only if $\eta = \nabla w=0$. Indeed, if $\eta = \nabla w$, then $\eta = \E \eta = \E [\nabla w ]= 0$. In the case $\eta = 0$ and $\nabla w \neq 0$ the term after $(i)$ in the above calculations equals zero and can be omitted. For equality $(ii)$ we used the one-to-one transformation $w \mapsto |\eta| w$ of $\D$.
\end{proof}

Propositions \ref{prop:first_equality} and \ref{t:alternativerep1} together with the following result show \eqref{eq:toshow_for_equality}.

\begin{proposition}\label{prop:finalequality}
Let $V$ be a potential and let $\F \in \FFF$ such that $\F \subset \F_\subw^2$. Assume $\inf_{f \in \F} K(f)>0$, then
\begin{align}
\inf_{f \in \F} \sup_{\eta \in S^{d-1}}
\left[ 2 K(f) \frac{\langle y,\eta\rangle^2}{H(\eta, f)} \right]^{1/2} = \sup_{\eta \in S^{d-1}} \inf_{f \in \F} 
\left[ 2 K(f) \frac{\langle y,\eta\rangle^2}{H(\eta, f)} \right]^{1/2}.\label{eq:equality}
\end{align}
\end{proposition}

\begin{proof}
One estimate is obvious. If $y=0$ \eqref{eq:equality} holds trivially. Without restriction we assume $y \in S^{d-1}$. 

We introduce for $f \in \F_\subw$,
\begin{align}\label{eq:defJ}
J(f):= \inf_{\eta \in S^{d-1}}  H(\eta, f)/ \langle y,\eta\rangle^2.
\end{align}
One has $J(f)>0$. Indeed, $H(\cdot,f)$ is continuous with respect to the topology on $S^{d-1}$, see \eqref{eq:Hcontinuous}. $S^{d-1}$ is compact and $H(\cdot,f)$ attains its infimum on $S^{d-1}$ in some point $\eta_0$ with $H(\eta_0,f)>0$, see \eqref{e:bounds_on_HK}. This shows
$J(f) \geq H(\eta_0,f) \inf_{\eta \in S^{d-1} } 1/\langle y, \eta \rangle^2 = H(\eta_0,f) >0$.

Throughout the following let $\F \in \FFF$ such that $\F \subset \F_\subw^2$. One has for $l \in \R$,
\begin{align}\label{eq:simplifiedcorresp2}
|l| \leq \inf_{f \in \F}\{K(f) / J(f)\}^{1/2}
\text{ if and only if }
0 \leq \inf_{f \in \F}\{ K(f) - l^2 J(f)\}.
\end{align}
Denote by $Q$ the set of $\eta \in S^{d-1}$ which are perpendicular to $y$. Since $H(\cdot,f)$ is positive on $S^{d-1}$, \eqref{eq:simplifiedcorresp2} is equivalent to
\begin{align}\label{e:tointerchange_minmax}
0 \leq \inf_{f \in \F} \sup_{\eta \in S^{d-1}\setminus Q} \{ K(f) - l^2  H(\eta,f)/\langle y, \eta\rangle^2 \}.
\end{align}

We are going to interchange $\inf$ and $\sup$ with the help of a minimax theorem, see \cite[Theorem 4.1']{Sion1958}, which we state here.

\begin{theorem} \label{t:concaveconvex}
Let $M$ and $N$ be any spaces, $F$ a function on $M \times N$ that is concave-convexlike. If for any
$c>\sup_{\mu \in M} \inf_{\nu \in N} F(\mu,\nu)$
there exists a finite subset $Y \subset N$ such that for any $\mu \in M$ there is a $y \in Y$ with $F(\mu,y)< c$, then
\begin{align*}
\sup_{\mu \in M} \inf_{\nu \in N} F(\mu,\nu) = \inf_{\nu \in N} \sup_{\mu \in M} F(\mu,\nu).
\end{align*}
\end{theorem}
A function $F$ on $M\times N$ is defined to be concavelike in $M$ if for any $\mu_1, \mu_2 \in M$ and $0\leq t \leq 1$ there exists a $\mu \in M$ such that for all $\nu \in N$,
\begin{align*}
tF(\mu_1,\nu) +(1-t) F(\mu_2,\nu) \leq F(\mu,\nu).
\end{align*}
Convexlike is defined analogously and a concave-convexlike function is concavelike in the first component and convexlike in the second.

We apply Theorem \ref{t:concaveconvex} to $M:=S^{d-1}\setminus Q$, $N:= \F$ and the function $F: M\times N \to \R$,
\begin{align*}
F(\eta,f) := K(f) - l^2 H(\eta,f)/\langle y,\eta \rangle^2.
\end{align*}
The interchange of infimum and supremum is trivial if $l=0$. Thus let $l \neq 0$. From the definition of $I$ it follows that $f \mapsto I(f)$ is convex. We apply the formula for $I(f)$, $f \in \F_\subw^2$, established in Proposition \ref{prop:formulaforI} to the vector $\lambda$ in direction of $\eta \in S^{d-1}\setminus Q$ with norm $|\lambda|=l\sqrt{2} /|\langle y, \eta \rangle|$ and get that the functional $I$ for $L^\lambda = (1/2)\Laplace + \lambda \nabla$ and $f \in \F_\subw^2$ is given by
\begin{align*}
I(f) 
& =\E\left[\frac{|\nabla f|^2}{8f}\right] + \frac{l^2}{\langle y, \eta \rangle^2} - \frac{l^2 H(\eta, f)}{\langle y, \eta\rangle^2},
\end{align*}
use \eqref{eq:H(lambda)=lambda^2H}. Therefore, $F(\eta,f) = I(f) - l^2 /\langle y, \eta \rangle^2 + \E[Vf]$ and $F$ is convex in $f$ for any $\eta \in M$. In particular, $F$ is convexlike in $N$.

In order to show that $F(\eta,f)$ is concavelike in $M$, choose $\eta_1$ and $\eta_2 \in S^{d-1}\setminus Q$, let $t \in [0,1]$. For $i=1,2$ set
\begin{align*}
\lambda_i:= \frac {\eta_i}{\langle y,\eta_i\rangle}
\end{align*}
which is well defined since $\eta_i \not\perp y$. Note that $t \lambda_1 + (1-t)\lambda_2 \neq 0$, since $\lambda_1 \in \spanning(\lambda_2)$ occurs only if $\lambda_1 = \lambda_2$. Hence, we can further choose
\begin{align*}
\eta_0:=\frac{t\lambda_1 + (1-t)\lambda_2}{|t\lambda_1 + (1-t)\lambda_2|}.
\end{align*}
One has
\begin{align}\label{eq:xieta}
\langle \eta_0 , y \rangle = 1/ |t \lambda_1 + (1-t)\lambda_2)|,
\end{align}
and $\eta_0 \notin Q$. By triangle inequality \eqref{e:triangle} and using convexity of $x \mapsto x^2$,
\begin{align*}
H(t\lambda_1 + (1-t)\lambda_2,f) 
& \leq t H( \lambda_1,f) + (1-t) H(\lambda_2,f).
\end{align*}
With \eqref{eq:xieta} we get for all $f \in \F$,
\begin{align*}
H(\eta_0,f) /\langle y, \eta_0 \rangle^2 
& = H(t\lambda_1 + (1-t)\lambda_2,f)\\
& \leq t H( \eta_1,f)/\langle y,\eta_1\rangle^2 + (1-t) H(\eta_2,f)/\langle y,\eta_2\rangle^2,
\end{align*}
thus $F$ is concavelike in the first component.

Recall that for $f$ fixed, $H(\cdot,f)$ restricted to $S^{d-1}$ is continuous, see \eqref{eq:Hcontinuous}. Hence, $F(\cdot,f)$ is continuous on $S^{d-1}\setminus Q$. We have $l>0$, and we may extend $F(\cdot,f)$ to a continuous function $\bar F(\cdot,f):S^{d-1} \to \R \cup \{-\infty\}$ by defining it to be $-\infty$ on $Q$. In order to see this, recognise that the uniform lower bound $H(\cdot,f)\geq\ess \inf_\P f$ assures for $C>0$,
\begin{align*}
\{ \eta \in S^{d-1}:\, H(\eta,f)/\langle \eta,y\rangle^2 >C \} 
\supset \{ \eta \in S^{d-1}:\, \ess \inf_\P f / \langle \eta,y\rangle^2 >C  \}.
\end{align*}
The latter is open in $S^{d-1}$. This implies continuity of $\bar F(\cdot,f)$ at any $\eta \in Q$, use \cite[Definition 7.1]{Willard1970}. Let
\begin{align*}
c>\sup_{\eta \in S^{d-1}\setminus Q} \inf_{f \in \F} F(\eta,f).
\end{align*}
We conclude as in \cite[Theorem 4.2]{Sion1958}: For each $f$ define $A_f:=\{\eta \in S^{d-1}: \bar F(\eta, f) <c\}$. The sets $A_f$, $f \in \F$, are open subsets of $S^{d-1}$ and cover $S^{d-1}$. Since $S^{d-1}$ is compact, there exists a finite subset $Y \subset \F$ such that the $A_f$, $f \in Y$, cover $S^{d-1}$. Hence, for any $\eta \in S^{d-1} \setminus Q$ there exists a $f \in Y$ such that $\bar F(\eta, f) = F(\eta, f) <c$.

Consequently, we can apply Theorem \ref{t:concaveconvex} to \eqref{e:tointerchange_minmax} and get
\begin{align}\label{e:final_zero_equation}
0
&\leq \inf_{f \in \F} \sup_{\eta \in S^{d-1}\setminus Q} \{ K(f) - l^2  H(\eta,f)/\langle y, \eta\rangle^2 \}\nonumber\\
&= \sup_{\eta \in S^{d-1}\setminus Q} \inf_{f \in \F}  \{ K(f) - l^2  H(\eta,f)/\langle y, \eta\rangle^2 \}.
\end{align}

As in \cite[Lemma 4.5]{Schroeder88} we deduce
  
\begin{lemma}
Assume $\inf_{f\in \F} K(f)>0$, then \eqref{e:final_zero_equation} implies
\begin{align}\label{eq:l_leq_supetainff}
|l| \leq \sup_{\eta \in S^{d-1}\setminus Q} \inf_{f \in \F}  [ K(f) \langle y, \eta\rangle^2 / H(\eta,f) ]^{1/2}.
\end{align}
\end{lemma}

\begin{proof}
By \eqref{e:final_zero_equation} for any $\epsilon>0$ exists $\eta \in S^{d-1}\setminus Q$, such that for $f \in \F$,
\begin{align*}
-\epsilon <  K(f) - l^2  H(\eta,f)/\langle y, \eta\rangle^2 .
\end{align*}
That is, for $f \in \F$ one has $l^2 < (K(f) +\epsilon)\langle y, \eta\rangle^2 /H(\eta,f)$. Since $\tilde \sigma(0):=\inf_{f \in \F} K(f)>0$,
\begin{align*}
l^2 < (1+\epsilon/\tilde\sigma(0))K(f)\langle y, \eta\rangle^2/H(\eta,f).
\end{align*}
This implies
$l^2 \leq (1+\epsilon/\tilde\sigma(0)) \inf_{f \in \F}\{K(f)\langle y, \eta\rangle^2/H(\eta,f)\}$.
Thus,
\begin{align*}
l^2 \leq (1+\epsilon/\tilde\sigma(0)) \sup_{\eta \in S^{d-1}\setminus Q}\inf_{f \in \F}\{K(f)\langle y, \eta\rangle^2/H(\eta,f)\}.
\end{align*}
Since $\epsilon>0$ is chosen arbitrarily, \eqref{eq:l_leq_supetainff} follows.
\end{proof}

\eqref{eq:l_leq_supetainff}, \eqref{e:final_zero_equation} together with \eqref{eq:simplifiedcorresp2} prove Proposition \ref{prop:finalequality}.
\end{proof}

This completes the proof of \eqref{eq:toshow_for_equality}.
The previous argument also shows that $\Gamma_V$ solves a variational equation: Recall the definition of $K$ and $J$ given in \eqref{eq:defK} and \eqref{eq:defJ} respectively.

\begin{proposition}\label{prop:ga_sol_of_minimizationproblem}
Let $V$ be a potential in $L^2$. Let $\F \in \FFF$, $\D \in \DDD$ and $y \neq 0$. Assume $\sigma(0)>0$. Then $\Gamma_V(y)$ is the unique nonnegative solution $l$ of
$0 = \inf_{f \in \F} \{ 2 K(f) - l^2 J(f) \}$.
\end{proposition}

\begin{proof}
Set $\bar M (l):=\inf_{f \in \F} \{ K(f) - l^2 J(f)\}$. $\bar M$ is decreasing and concave on $[0,\infty)$, which follows as in \eqref{eq:M_concave} and \eqref{eq:M_decreasing}. The statement now results from Proposition \ref{t:alternativerep1}, \eqref{eq:simplifiedcorresp2} and $\bar M (0) = \sigma(0)>0$.
\end{proof}

\subsection{Almost Sure Equality on whole \texorpdfstring{$\R^d$}{Rd}}
\label{subsec:equalityonwholeR}

Let $(\Omega,\FF,\P,\tau)$ be an ergodic dynamical system and $V$ be a \pott{}. We have shown in subsections \ref{subsec:upperbound}, \ref{subsec:lowerbound} and \ref{subsec:equivalence} that for any $y \in \R^d$ $\P$-a.s.\ $\Gamma_V(y)=\RR(-\Lambda_\omega)(y)$. The stronger statement that equality holds $\P$-a.s.\ for any $y$ is implied by the following: $V$ is bounded and we have $0 \geq \Lambda_\omega(\lambda) \geq -\vmax$. Hence, $\sup\{|\lambda|:\, -\Lambda_\omega(\lambda)-\lambda^2/2 > 0\}<\infty$ and we get by Lemma \ref{lem:inverseofR} continuity of $\RR(-\Lambda_\omega)$. In the same way $\Gamma_{V}=\RR(\sigma)$ is continuous, use \eqref{eq:M_concave} and estimate the infimum there by choosing $f\equiv 1$ to see $\sigma(\lambda)-\lambda^2/2 \leq \E[V] - \lambda^2/2$, where we used \eqref{eq:H(lambda)=lambda^2H} and \eqref{e:bounds_on_HK}. Therefore, equality holding $\P$-a.s.\ on a dense subset of $\R^d$ ensures $\Gamma_{V}=\RR(-\Lambda_\omega)$ on whole $\R^d$.

\section{Appendix}\label{app:appendix}

\subsection{Condition (E1)}\label{app:homogenization}

Condition (E1) is valid as soon as homogenization takes place and the effective Hamiltonian has a variational expression as outlined in this subsection. For an overview see also \cite{Kosygina07}. Throughout the following let $\F:=\F_\subw^2$ and $\D:=\D_\subw^2$.

Let $\lambda \in \R^d$. For $(t,x,\omega) \in [0,\infty)\times\R^d\times\Omega$ and $\epsilon>0$ define
\begin{align*}
u_\epsilon(t,x,\omega) := \epsilon \ln E_{x/\epsilon}^\lambda\bigg[\exp\bigg\{-\int_0^{t/\epsilon} V_\omega(Z_s)ds\bigg\}\bigg].
\end{align*}
Assuming a Feynman-Kac correspondence, see the comment after Proposition \ref{prop:logmomgenfuncupperbound}, $u_\epsilon$ solves the Hamilton-Jacobi-Bellman equation
\begin{align}\label{eq:HJB}
\partial_t u_\epsilon(t,x,\omega) 
&= \frac{\epsilon}{2} \Laplace u_\epsilon(t,x,\omega) + \hamil^\lambda\bigg(\nabla u_\epsilon (t,x,\omega), \frac{x}{\epsilon},\omega\bigg)
\end{align}
with initial condition $u_\epsilon (0,\cdot,\omega) \equiv 0$, and Hamiltonian $\hamil^\lambda: \R^d \times \R^d \times \Omega \to \R$ given as $\hamil^\lambda(p,x,\omega):= (1/2)p^2 + \lambda p - V_\omega(x)$. Let $\lagrange^\lambda(q,x,\omega) 
:= (1/2)(q-\lambda)^2 + V_\omega(x)
$, the convex conjugate of $\hamil^\lambda(\cdot,x,\omega)$.

In \cite{Kosygina2006} it is shown that homogenization of \eqref{eq:HJB} takes place: For $\lambda \in \R^d$ $\P$-a.s.,
\begin{align}\label{ass:Lambda=effectiveHamiltonian}
\lim_{t \to \infty} \frac{1}{t}\ln E^\lambda\bigg[\exp\bigg\{-\int_0^tV_\omega(Z_s)ds\bigg\}\bigg]
= \bar\hamil^\lambda(0),
\end{align}
where $\bar\hamil^\lambda:\R^d \to \R$ is the effective Hamiltonian $\bar\hamil^\lambda(p) := \sup_{(b,f) \in \EE}\E[(pb - \lagrange^\lambda(b_\omega,0,\omega))f]$, with $\B:=L^\infty$ and
\begin{align}\label{e:defofE}
\EE:=\bigg\{(b,f)\in \B \times \F:\, \frac{1}{2}\Laplace f = \nabla (bf)\bigg\}.
\end{align}
Equation $(1/2)\Laplace f = \nabla(bf)$ has to be interpreted in the `distributional sense' on $\R^d$, see \cite[(6.2)]{Kosygina07}, i.e.\ $\P$-a.s.\ for any $\varphi \in C_\compactc^\infty$ one has $\int ((1/2)\Laplace f_\omega) \varphi dx = - \int b_\omega f_\omega \nabla\varphi dx$. By Lemma \ref{lem:weakdivergence} any $(b,f) \in \EE$ satisfies for all $w \in \D$,
\begin{align}\label{e:defofE_doesHoldForD}
\int\bigg(\frac{1}{2} \Laplace w + b \nabla w \bigg)fd\P = 0.
\end{align}

We have the following estimate on the effective Hamiltonian:

\begin{proposition}\label{prop:logmomgenfuncupperbound}
For all $\lambda \in \R^d$ one has $\bar\hamil^\lambda(0) \leq -\sigma(\lambda)$.

\end{proposition}

\begin{proof}
We estimate $\bar H^\lambda(0)$ similarly to \cite[(5.2)-(5.6)]{Kosygina2006}.
\begin{align*}
\sup_{(b,f)\in\EE}\E[-\lagrange^\lambda(b_\omega,0,\omega)f]
& \overset{(i)}{=} \sup_{f\in\F} \sup_{b\in\B} \inf_{w\in\D} \int \bigg( -\lagrange^\lambda(b_\omega,0,\omega) + \frac{1}{2}\Laplace w + b \nabla w \bigg)f d\P\\
& \overset{(ii)}{\leq} \sup_{f\in\F} \inf_{w\in\D} \int \sup_{b\in\B} \bigg\{\bigg( -\lagrange^\lambda(b_\omega,0,\omega) + \frac{1}{2}\Laplace w + b \nabla w \bigg)f\bigg\}d\P\\
& \overset{(iii)}{=} \sup_{f\in\F} \inf_{w\in\D} \int\bigg( \frac{1}{2}\Laplace w + \hamil^\lambda(\nabla w,0,\omega)\bigg)fd\P.
\end{align*}
$(i)$ is valid by \eqref{e:defofE_doesHoldForD}, in order to interchange $\sup_{b\in\B}$ with integration in $(ii)$ note that measurability of the $\sup$ over the integrand is guaranteed by the subsequent calculation of the integrand. In $(iii)$ we used the definition of $\lagrange^\lambda$ as the convex conjugate of $\hamil^\lambda$ and the fact, that the convex biconjugate of $\hamil^\lambda(\cdot,0,\omega)$ is again $\hamil^\lambda(\cdot,0,\omega)$, see \cite[Lemma 4.5.8]{Dembo98}. We continue inserting the definition of $\hamil^\lambda$.
\begin{align}
&= \sup_{f\in\F} \inf_{w\in\D} \int \bigg(\frac{1}{2}\Laplace w + \frac{1}{2} |\nabla w|^2 + \lambda \nabla w - V\bigg)fd\P \label{eq:sigma_intermsof_hamil}
= \sup_{f\in\F} \inf_{u\in\U} \int \left(\frac{Lu}{u} - V\right)fd\P,
\end{align}
since $e^\D=\U$, and $\Laplace e^w = (\nabla w)^2 e^w + (\Laplace w) e^w$, see \cite[(7.18) and Theorem 7.8]{Gilbarg83}.
\end{proof}

Therefore, \eqref{ass:Lambda=effectiveHamiltonian} is sufficient for condition (E1) to be valid. There is extensive research concerning homogenization. For example, \cite[Theorem 2.3]{Kosygina2006}, \cite[Theorem 3.1(i)]{Lions2005} or \cite[Theorem 1]{Armstrong2012} exhibit the desired convergence. In \cite{Kosygina2006} the effective Hamiltonian is given as required. \cite[Theorem 2.3]{Kosygina2006} applies if the potential is bounded, uniformly continuous uniformly for all $\omega \in \Omega$, and \cite[(3.1)]{Kosygina2006} needs to be valid. \cite[(3.1)]{Kosygina2006} is an optimal control problem and we refer to \cite[Paragraphs IV.3. and IV.4]{Fleming2006} for a discussion of solvability. \eqref{eq:HJB} is mainly a Feynman-Kac correspondence. Sufficient criterions for a Feynman-Kac correspondence are given in \cite[Paragraph 6.5]{Friedman1975}.

\subsection{Free Energy}\label{app:freeenergy}

In this subsection we deduce the identification \eqref{eq:identification} and the variational expression for the quenched free energy stated in Corollary \ref{cor:quenchedfreeenergy}. For $a:\R^d \to \R$ we introduce $\bar \RR: \R^d \to \R \cup \{ \pm \infty\}$,
\begin{align*}
\bar \RR(a)(y):= \sup \{y\lambda:\, \lambda \in \R^d,\, a(\lambda)-\lambda^2/2 \geq 0\}.
\end{align*}
Of course, $\RR(a) \leq \bar \RR(a)$. We calculate the inverse of $\RR$, see also \cite[(8.2)]{Armstrong2012}.

\begin{lemma}\label{lem:inverseofR}
Let $a:\R^d \to \R$. If $\lambda \mapsto a(\lambda)-\lambda^2/2$ is concave and $a(0)>0$, then $\RR(a) = \bar\RR(a)$, and for $c \leq a(0)$, for $\lambda \in \R^d$,
\begin{align}\label{eq:varform_for_sigma}
(a(\lambda)-\lambda^2/2) \wedge c = \sup\{-\mu:\, -\mu < c,\, \lambda y \leq \RR(a+\mu)(y)\, \forall y \in \R^d\}.
\end{align}
If $C:=\sup\{|\lambda|:\, a(\lambda)-\lambda^2/2 > 0\} < \infty$, then $\RR(a)$ is Lipschitz continuous with Lipschitz constant $C$.
\end{lemma}

\begin{proof}
Let $M(\lambda):=a(\lambda)-\lambda^2/2$. In order to see the first statement use as in \eqref{eq:R=supsup} the fact that 
$\RR(a)(y) = \sup_{\eta \in S^{d-1}} \sup_{s \geq 0:\, M(s\eta)>0} sy\eta$. The hypotheses on $a$ show $\RR(a) = \bar\RR(a)$.
We have for $-\mu<c$,
\begin{align}\label{eq:subdifferential}
\{\lambda \in \R^d:\, M(\lambda) \geq -\mu\}
=\{\lambda \in \R^d:\, y \lambda \leq \RR(a+\mu)(y)\, \forall y \in \R^d\}.
\end{align}
In fact, $\subset$ follows directly, use $\RR(a) = \bar\RR(a)$. For the reverse, let $\bar \lambda$ such that $y \bar\lambda \leq \RR(a+\mu)(y)$ for $y \in \R^d$. Assume $M(\bar\lambda)<-\mu$. Since $M$ is continuous and $M(0)>-\mu$, there exists $0<t<1$ maximal such that $M(t\bar\lambda)=-\mu$. We have that the level set $S_\mu$ $:=$ $\{\lambda \in \R^d:$ $M(\lambda) \geq -\mu\}$ is convex, see \cite[Theorem 4.6]{Rockafellar70}. Moreover, $t\bar\lambda$ is in the boundary of $S_\mu$. Thus, there exists a supporting hyperplane $T$ of $S_\mu$ through $t\bar\lambda$, see \cite[Corollary 11.6]{Rockafellar70}. In particular, there is $\bar y \neq 0$, $\bar y \perp T$ and
\begin{align*}
\bar\RR(a+\mu)(\bar y)
=\sup\{\bar y \lambda:\, \lambda \in S_\mu\}
=\sup\{\bar y \lambda:\, \lambda \in T\}
= \bar y t \bar \lambda,
\end{align*}
which is in contrary to $\bar y \bar \lambda \leq \RR(a+\mu)(\bar y)$. This shows \eqref{eq:subdifferential}.
\eqref{eq:varform_for_sigma} follows from \eqref{eq:subdifferential}, and from the fact that $(a(\lambda)-\lambda^2/2)\wedge c
= \sup\{-\mu:\, -\mu<c,\, a(\lambda)-\lambda^2/2 \geq -\mu\}$.

Let $C<\infty$. Then $\RR(a)(y+z) \leq \RR(a)(y) + \RR(a)(z)\leq \RR(a)(y)+C|z|$. For the lower bound use $\langle y+z,\lambda \rangle \geq \langle y,\lambda \rangle - |z||\lambda|$, hence,
$\RR(a)(y+z) \geq \sup\{y\lambda - |z|C:\, \lambda \in \R^d,\,a(\lambda)-\lambda^2/2 > 0\} = \RR(a)(y)-C|z|$,
which shows continuity.
\end{proof}

For $\lambda \in \R^d$ let $\tilde \hamil^\lambda(\cdot)$ be the effective Hamiltonian for the homogenization problem \eqref{eq:HJB} given by \cite[Theorem 1]{Armstrong2012}, if applicable. That is, $\P$-a.s.\ $u_\epsilon \to u$ locally uniformly with $u$ solution to $\partial_t u = \tilde \hamil^\lambda(\nabla u)$, $u(0,\cdot)\equiv 0$.

\begin{proposition}\label{prop:identification}
\EDS{} Let $V$ be a \pott{}. Assume there exists $C>0$ such that for $\omega \in \Omega$, for $x_1$, $x_2 \in \R^d$, one has $|V_\omega(x_2)-V_\omega(x_1)|\leq C |x_2-x_1|$.
Then \eqref{eq:identification} is valid.
\end{proposition}

\begin{proof}
The assumptions on the potential ensure that we can apply \cite[Theorem 1]{Armstrong2012}: As in \eqref{ass:Lambda=effectiveHamiltonian} with help of a Feynman-Kac correspondence \eqref{eq:HJB} we get for any $\lambda \in \R^d$ $\P$-a.s.\ $\Lambda_\omega(\lambda) = \tilde H^\lambda(0)$. Moreover, we have $\P$-a.s.\ for all $\lambda \in \R^d$,
\begin{align}\label{eq:Lambda=tildeH}
\Lambda_\omega(\lambda) = \tilde H^\lambda(0).
\end{align}
Indeed, recall that $\Lambda_\omega(\lambda)+\lambda^2/2$ is convex and continuous in $\lambda$. On the other hand, for $p,\lambda \in \R^d$ we have $\tilde \hamil^\lambda(p)+\lambda^2/2=\tilde \hamil^0(p+\lambda)$, use e.g.\ \cite[(5.35)]{Armstrong2012}. Thus, continuity of $\tilde H^\lambda(0) + \lambda^2/2= \tilde H^0(\lambda)$ follows from \cite[Proposition 5.5]{Armstrong2012}. This shows \eqref{eq:Lambda=tildeH}.

By assumptions we have $-\Lambda_\omega(0)\geq \sigma(0)>0$ $\P$-a.s.. Lemma \ref{lem:inverseofR} thus shows $\RR(-\Lambda_\omega) = \bar \RR(-\Lambda_\omega)$. Therefore, Theorem \ref{t:varform} and \cite[Lemma 6.7]{Armstrong2012} imply $\P$-a.s.\ for $y \in \R^d$,
\begin{align}\label{eq:Gamma=barm}
\Gamma_{V}(y)
= \bar \RR(-\Lambda_\omega)(y)
= \bar \RR(-\tilde \hamil^\cdot(0))(y) 
= \bar m(-y),
\end{align}
where $\bar m$ is the unique solution to the effective metric problem, see \cite[(6.24)]{Armstrong2012},
\begin{align*}
\tilde \hamil^0(-\nabla \bar m) = 0 \text{ in } \R^d\setminus\{0\},\
\bar m(0)=0, \ 
\liminf_{|y|\to \infty} |y|^{-1} \bar m(y)\geq 0.
\end{align*}
Since $- \ln e(y,\omega)$ solves the associated metric problem, use \cite[Proposition 2.3.8]{Sznitman98}, as in \cite[(8.1)]{Armstrong2012} with \cite[Proposition 6.9]{Armstrong2012} and \eqref{eq:Gamma=barm} the identification \eqref{eq:identification} follows.
\end{proof}

\begin{proof}[Proof of Corollary \ref{cor:quenchedfreeenergy}]
Let $\mu_0 > 0$, $\mu_0 \in \Q$. With Theorem \ref{t:varform} we have $\P$-a.s.\ for any $\mu > 0$, $\mu \in \Q$,
\begin{align}\label{eq:1:cor:quenchedfreeenergy}
\RR(\sigma + \mu)= \RR(-\Lambda_\omega+\mu).
\end{align}
Recall that $\lambda \mapsto \sigma(\lambda)-\lambda^2/2$ is concave, see \eqref{eq:M_concave}.  Let $c \leq \min\{-\Lambda_\omega(0)+\mu_0,\, \sigma(0)+\mu_0\}$ $\P$-a.s.. We use \eqref{eq:varform_for_sigma} and get for $\lambda \in \R^d$,
\begin{align*}
(\sigma(\lambda)+\mu_0 - \lambda^2/2)\wedge c
& = \sup\{-\mu \in \R:\, -\mu < c,\, \lambda y \leq \RR(\sigma+\mu_0+\mu)(y)\, \forall y \in \R^d\}.
\end{align*}
The function $\mu \mapsto  \RR(\sigma+\mu)(y)$ is monotone increasing, and we can reduce the supremum in the following to $\mu \in \Q$ and obtain by \eqref{eq:1:cor:quenchedfreeenergy} that the latter $\P$-a.s.\ for any $\lambda \in \R^d$ equals
\begin{align*}
\sup\{-\mu \in \Q:\, -\mu < c,\, \lambda y \leq \RR(-\Lambda_\omega+\mu_0+\mu)(y)\, \forall y \in \R^d\}
& = (-\Lambda_\omega(\lambda) + \mu_0-\lambda^2/2)\wedge c.
\end{align*}
The last equality is obtained by \eqref{eq:varform_for_sigma} and by convexity of $\lambda \mapsto \Lambda_\omega(\lambda)+\lambda^2/2$.
\end{proof}

\subsection{Denseness} \label{app:derivative}

In order to prove the denseness results stated in Section \ref{sec:preliminaries} we use the notion of convolution, see \cite[(7.19)]{Jikov1994}: For $f \in L^1$ the convolution of $f$ with $\g \in C_\compactc^\infty$ is defined as
$f \ast \g (\omega) := \int f_\omega(x) \g(x)dx$, where $\omega$ is in a set of full $\P$-measure such that $f_\omega \in L^1_\loc$. For an even function $\g \in C_\compactc^\infty$ such that $\int \g(x) dx = 1$, $\g \geq 0$, define for $x \in \R^d$ and $\epsilon>0$ the function $\g_\epsilon(x):= \epsilon^{-d} \g(x/\epsilon)$. In the following, whenever we write $f_\epsilon$ we mean $f\ast\g_\epsilon$ where $\g$ has the prescribed properties. If $\phi \in (L^2)^d$ we also define $\phi \ast \g := (\phi_i \ast \g)_i$.
The following properties of the convolution are needed, see e.g.\ \cite[Chapter 7]{Jikov1994}:

\begin{lemma}\label{lem:convproperties}
Let $f \in L^1$ and $\g \in C_\compactc^\infty$. Then $f \ast \g$ is differentiable of all orders. Moreover,
\begin{align}
\label{eq:youngs_inequality}
&\|f \ast \g\|_p \leq \|f\|_p\|\g\|_1, &\text{ for } f \in L^p,\ p \geq 1,\\ 
\label{eq:derivativecommuteswithconv}
&\partial_i(f \ast \g)  = (\partial_i f) \ast \g, &\text{ for } f \in \DD(\partial_i),\\
\label{eq:derivativeofconv_bdd}
& \sup{_\Omega}|D^n(f\ast \g)| <\infty,\ n\in \N_0, &\text{ if $\sup{_\Omega} |f|<\infty$},\\
\label{eq:convolutionconverges}
&\|f-f_\epsilon\|_2 \to 0 \text{ as $\epsilon \searrow 0$}, &\text{ for $f \in L^2$.}
\end{align}
\end{lemma}

The space $\D_\subs$ serves as space of test functions:

\begin{lemma}\label{lem:testfunctions}
$\D_\subs$ is dense in $L^2$.
\end{lemma}

\begin{proof}
Let $\phi \in L^2$. For $M>0$ set $\phi_M := \phi1_{|\phi|\leq M}$. $\sup_\Omega|\phi_M| < \infty$ and $\phi_M \to \phi$ as $M \to \infty$ with respect to $\|\cdot\|_2$. Let $\delta >0$ and choose $M$ such that $\|\phi_M - \phi\|_2 < \delta/2$. \eqref{eq:convolutionconverges} allows to choose $\epsilon >0$ such that $\|(\phi_M)_\epsilon - \phi_M\|_2 < \delta/2$. Hence, $\|(\phi_M)_\epsilon - \phi\|_2<\delta$ and $(\phi_M)_\epsilon \in \D_\subs$, see \eqref{eq:derivativeofconv_bdd}.
\end{proof}

A vector field $\psi \in (L^2)^d$ is called divergence-free or solenoidal, if $\div\psi = 0$ $\P$-a.s., where for those $\omega$ for which $\psi_\omega \in L^1_\loc$ one defines $(\div\psi)_\omega: C_\compactc^\infty \to \R$,
\begin{align*}
(\div\psi)_\omega(\varphi) := -\sum_i \int \psi_{\omega,i} \partial_i\varphi dx.
\end{align*}
A vector field $\psi \in (L^2)^d$ is called rotation-free or potential, if $\rot\psi = 0$ $\P$-a.s., where for those $\omega$ for which $\psi_\omega \in L_\loc^1$ one defines $(\rot\psi)_\omega: C_\compactc^\infty \to \R^{\binom{d}{2}}$, and for $i<j$,
\begin{align*}
\left((\rot\psi)_\omega(\varphi)\right)_{i,j}
:= \int \psi_{\omega,i} \partial_j\varphi - \psi_{\omega,j} \partial_i\varphi dx.
\end{align*}
We recall the orthogonal decomposition of $(L^2)^d$ into potential and solenoidal vector fields, called Weyl's decomposition, see \cite[Lemma 7.3]{Jikov1994}:
\begin{align}\label{eq:weyl}
(L^2)^d=\VV_\pot \oplus \VV_\sol \oplus \R^d,
\end{align}
where $\VV_\pot = \{\phi \in (L^2)^d: \rot\phi = 0,\, \E \phi = 0\}$ and $\VV_\sol  = \{\phi \in (L^2)^d: \div\phi = 0,\, \E \phi = 0\}$.

The following lemma lifts the notion of `weak divergence' to $(L^2)^d$:

\begin{lemma}\label{lem:weakdivergence}
Assume $\phi \in (L^2)^d$ and $h \in L^2$. Then $\P$-a.s.\ for any $\varphi \in C_\compactc^\infty$,
\begin{align}\label{eq:weakdivergence:1}
-\int \phi_\omega \nabla \varphi dx = \int h_\omega \varphi dx,
\end{align}
if and only if for any $w \in \D_\subw$,
\begin{align}\label{eq:weakdivergence:2}
- \langle \phi , \nabla w \rangle = \langle h , w \rangle.
\end{align}
The second statement is equivalent to \eqref{eq:weakdivergence:2} forced to hold for some space $\D \subset \D_\subw$ which is invariant under $\tau_x$, $x\in\R^d$, and dense in $L^2$ with respect to $\|\cdot\|_2$.
\end{lemma}

As a consequence of Lemma \ref{lem:weakdivergence}, choosing $h \equiv 0$, we have
\begin{align}\label{eq:phi_characterization}
\Phi_y^\supw = y + \VV_\sol.
\end{align}

\begin{proof}
Assume \eqref{eq:weakdivergence:1} and consider $\phi_\epsilon$ and $h_\epsilon$. Then $\P$-a.s.\ for $\epsilon>0$, for any $\varphi \in C_\compactc^\infty$,
\begin{align}\label{eq:regularized_nablaphi=h}
\int \nabla \phi_{\epsilon,\omega} \varphi dx
= \int h_{\epsilon,\omega} \varphi dx.
\end{align}
In fact, $\P$-a.s.\ for $\epsilon > 0$, for $\varphi \in C_\compactc^\infty$ by Fubini's theorem
\begin{align*}
- \int \phi_{\epsilon,\omega}(x) \nabla \varphi(x) dx
&= - \sum_i \int \int \phi_{i,\omega}(x) \partial_i \varphi(x-y) dx \, \g_\epsilon(y) dy
\end{align*}
which by \eqref{eq:weakdivergence:1} and since $\varphi(\cdot-y) \in C_\compactc^\infty$ equals
\begin{align*}
& = \int \int h_{\omega}(x) \varphi(x-y) dx \, \g_\epsilon(y) dy
= \int h_{\epsilon,\omega}(x) \varphi(x) dx.
\end{align*}
An elementary denseness argument shows that \eqref{eq:regularized_nablaphi=h} implies
$\P$-a.s.\ for $\epsilon>0$ $\leb$-a.e.\ $\nabla \phi_{\epsilon,\omega} = h_{\epsilon,\omega}$. In fact, $C_\compactc^\infty$ and $L^2_\loc$ are in duality, see \cite[Paragraph IV.1]{Schaefer1974}. Therefore, we have for $\epsilon>0$ $\P$-a.s.\ $\nabla \phi_\epsilon = h_\epsilon$. Since $\phi_{i,\epsilon} \to \phi_i$, $h_\epsilon \to h$ in $L^2$, we get for all $w \in \D_\subw$ by integration by parts, see \eqref{eq:partialintegration},
\begin{align*}
- \E[\phi \nabla w] 
& = \lim_{\epsilon \to 0} - \E[\phi_\epsilon \nabla w] 
= \lim_{\epsilon \to 0} \E[\nabla \phi_\epsilon w] 
= \lim_{\epsilon \to 0} \E[h_\epsilon w] 
= \E[hw],
\end{align*}
which was to show.

The proof of the reverse direction is analogous: Let $\epsilon>0$ and $\phi$ satisfy \eqref{eq:weakdivergence:2}. For $w \in \D$ it is
$\langle \nabla \phi_\epsilon , w \rangle = \langle h_\epsilon , w \rangle$.
Indeed, using \eqref{eq:partialintegration}, Fubini's theorem, invariance of $\P$, the fact that $w \circ \tau_x \in \D$, and that $(\partial_i w) \circ \tau_x = \partial_i(w \circ \tau_x)$ for any $x \in \R^d$,
\begin{align*}
\langle \nabla \phi_\epsilon , w \rangle
&= -\sum_i \int \E \left[\phi_i \partial_i ( w \circ \tau_{-x})\right] \g_\epsilon(x) dx
= \int \E \left[h(w \circ \tau_{-x})\right] \g_\epsilon(x)dx
= \langle h_\epsilon, w \rangle.
\end{align*}
$\D$ is dense in $L^2$, hence, for all $\epsilon >0$ $\P$-a.s.\ $\nabla \phi_\epsilon = h_\epsilon$.
$\P$-a.s.\ $\phi_{\epsilon,\omega,i} \to \phi_{\omega,i}$ and $h_{\epsilon,\omega} \to h_\omega$ in $L^2_\loc$. This shows along a sequence $(\epsilon_n)_n \subset [0,\infty)$ converging to zero, $\P$-a.s.\ for $\varphi \in C_\compactc^\infty$,
\begin{align*}
- \int \phi_{\omega} \nabla \varphi dx
& = -  \lim_{n} \int \phi_{\epsilon_n,\omega} \nabla \varphi dx
= \lim_{n} \int \nabla \phi_{\epsilon_n,\omega} \varphi dx
= \lim_{n} \int h_{\epsilon_n,\omega} \varphi dx
=\int h_{\omega} \varphi dx,
\end{align*}
which shows the statement.
\end{proof}

\begin{lemma}\label{lem:D_dense_in_V_sol}
Let $\D$ be a vector subspace of $L^2$ of functions which are differentiable of all orders, assume $\D$ is dense in $L^2$ such that $\partial_i \D \subset \D$ for any $i$ and $\tau_x \D \subset \D$ for any $x \in \R^d$.
Then $\D^d \cap \VV_\sol$ is dense in $\VV_\sol$. 
\end{lemma}

\begin{proof}
We introduce the space $Y_\sol$ of vector fields $\phi$ of the form
\begin{align}\label{eq:defYsol}
(\phi_k)_{1\leq k \leq d}=\bigg(\sum_{i:\, i<k} (-1)^{i-k-1} \partial_i w_{ik} + \sum_{i:\, i>k} (-1)^{i-k} \partial_i w_{ki}\bigg)_{1 \leq k \leq d},
\end{align}
where $w_{ik} \in \D$, $1 \leq i,k \leq d$. Since the differential operator is linear and since $\D$ is a vector space, $Y_\sol$ is a vector space itself.
Moreover, $Y_\sol$ is a subspace of $\D^d$ as well as of $\VV_\sol$. The second follows from a direct calculation of the divergence of vector fields $\phi \in Y_\sol$: $\P$-a.s.\
\begin{align*}
\div \phi 
= \nabla \phi 
&= \sum_{(i,k):\, 1 \leq  i<k \leq d} \bigg((-1)^{i-k-1} \partial_{ki} w_{ik} + (-1)^{k-i} \partial_{ik} w_{ik}\bigg)=0,
\end{align*}
where we used that $\partial_i\partial_k w = \partial_k \partial_i w$ for $w$ differentiable of all orders.

We are now going to show that $Y_\sol$ in fact is dense in $\VV_\sol$ which together with the fact that $Y_\sol \subset \D^d$ then shows the statement.

Let $\psi \in (Y_\sol)^\perp$. Since for any $\phi \in Y_\sol$ and $x \in \R^d$ also $\phi \circ \tau_x \in Y_\sol$, we get for $\epsilon>0$,
\begin{align*}
\langle \phi, \psi_\epsilon \rangle
= \E\bigg[\phi \int \psi_\omega(x) \g_\epsilon(x) dx\bigg]
& = \int \E[\phi_\omega(-x) \psi_\omega] \g_\epsilon(x) dx 
= \int \langle \phi \circ \tau_{-x} , \psi \rangle \g_\epsilon(x) dx = 0
\end{align*}
where we used Fubini's theorem and translation invariance of $\P$. Thus, for $\phi \in Y_\sol$,
\begin{align*}
0=\langle \phi,\psi_\epsilon \rangle
&=  \sum_{(i,k):\, i<k} (-1)^{i-k} \langle w_{ik} , \partial_i \psi_{\epsilon,k} - \partial_k \psi_{\epsilon, i} \rangle.
\end{align*}
By denseness of $\D$ in $L^2$, for $i<k$ one has $\partial_i \psi_{\epsilon,k} - \partial_k \psi_{\epsilon,i} = 0$. Therefore, for $\epsilon >0$ $\P$-a.s.
\begin{align*}
(\rot\psi_\epsilon)_\omega(\varphi) 
= \left(\int \psi_{\epsilon,\omega,i} \partial_k\varphi - \psi_{\epsilon,\omega,k} \partial_i \varphi dx\right)_{i<k}
= \left(\int(\partial_i \psi_{\epsilon,\omega,k} - \partial_k \psi_{\epsilon,\omega,i})\varphi dx\right)_{i<k} 
=0
\end{align*}
for any $\varphi \in C_\compactc^\infty$. Using that $\P$-a.s.\ in $L^2_\loc$ one has $\psi_{\epsilon,\omega} \to \psi_\omega$, we get $\P$-a.s.\ 
$\rot\psi = 0$.
Hence, $(Y_\sol)^\perp \subset \VV_\pot\oplus \R^d$. Weyl's decomposition \eqref{eq:weyl} leads to
$\overline{Y_\sol}=((Y_\sol)^\perp)^\perp \supset (\VV_\pot\oplus\R^d)^\perp
 = \VV_\sol$,
where $\overline{Y_\sol}$ is the closure of $Y_\sol$ in $L^2$. Note that even $\VV_\sol = \overline{Y_\sol}$.
\end{proof}

\begin{proof}[Proof of Lemma \ref{lem:spacesdense}]
Let $f$ in $\D_\subw$ and set $f_M:=f1_{|f|\leq M}$ for $M>0$. For $n \in \N$ with the help of \eqref{eq:convolutionconverges} and \eqref{eq:derivativecommuteswithconv} choose $\epsilon=\epsilon(n)$ such that $\|f_\epsilon - f\|_2 \leq 1/n$ and $\sum_i \|\partial_if_\epsilon - \partial_i f\|_2 \leq 1/n$. Using Young's inequality \eqref{eq:youngs_inequality} and the fact that $\partial_i(f_M)_\epsilon - \partial_i f_\epsilon = (f- f_M) \ast\partial_i\g_\epsilon$, we are able to choose $M=M(\epsilon) = M(n)$ such that
$\|(f_M)_\epsilon - f_\epsilon\|_2 \leq 1/n$ and $\sum_i \|\partial_i(f_M)_\epsilon - \partial_i f_\epsilon\|_2 \leq 1/n$. Then with \eqref{eq:derivativeofconv_bdd} $((f_{M(n)})_{\epsilon(n)})_n \subset \D_\subs$ converges in $\| \cdot \|_\nabla$ to $f$ for $n \to \infty$.

Let $f \in \F_\subw$ with $f\geq c>0$ $\P$-a.s.. Define $\bar f := (f \wedge \|f\|_\infty) \vee c$. Then $\bar f_\epsilon \in \F_\subs$ and approximates $f$ in the desired way as before.

Note that $\Phi_y^\supw=y+\VV_\sol$ and $\Phi_y^\sups = y+(\D_\subs)^d \cap \VV_\sol$, see \eqref{eq:phi_characterization}. The last statement therefore follows from Lemma \ref{lem:D_dense_in_V_sol} and Lemma \ref{lem:testfunctions}.
\end{proof}

\noindent
\thanks{\textbf{Acknowledgement:}
This work was funded by the ERC Starting Grant 208417-NCIRW.\\
I am glad to express my thanks to my supervisor Prof.\ Dr.\ Martin P.\ W.\ Zerner for proposing me this very interesting problem, for profitable discussions and for his support.
I thank Prof.\ Dr.\ Elena Kosygina who substantially contributed to the accomplishment of this project with her profound knowledge of the subject.
Many thanks go to Dr.\ Elmar Teufl who always supported this work with valuable advices and discussions.
I am grateful having had the opportunity to talk to Prof.\ Dr.\ Rainer Sch\"atzle and Prof.\ Dr.\ S.\ R.\ Srinivasa Varadhan.}

\small
\bibliographystyle{amsalpha}

\begin{thebibliography}{RASY13}

\bibitem[AS12]{Armstrong2012}
S.~N. Armstrong and P.~E. Souganidis, \emph{Stochastic homogenization of
  {H}amilton-{J}acobi and degenerate {B}ellman equations in unbounded
  environments}, J. Math. Pures Appl. (9) \textbf{97} (2012), no.~5, 460--504.
  \MR{2914944}

\bibitem[CL90]{Lacroix1990}
R.~Carmona and J.~Lacroix, \emph{Spectral theory of random {S}chr\"odinger
  operators}, Birkh\"auser Boston Inc., Boston, MA, 1990. \MR{1102675
  (92k:47143)}

\bibitem[DMM86]{DalMaso1986}
G.~Dal~Maso and L.~Modica, \emph{Nonlinear stochastic homogenization and
  ergodic theory}, J. Reine Angew. Math. \textbf{368} (1986), 28--42.
  \MR{850613 (88k:28021)}

\bibitem[Doo84]{Doob1984}
J.~L. Doob, \emph{Classical potential theory and its probabilistic
  counterpart}, vol. 262, Springer-Verlag, New York, 1984. \MR{731258
  (85k:31001)}

\bibitem[dT09]{DelTenno2009SpecialExamples}
I.~del Tenno, \emph{Special examples of diffusions in random environment},
  Stochastic Process. Appl. \textbf{119} (2009), no.~3, 924--936. \MR{2499864
  (2010i:60288)}

\bibitem[Dur96]{Durrett96}
R.~Durrett, \emph{Stochastic calculus}, CRC Press, Boca Raton, FL, 1996.
  \MR{1398879 (97k:60148)}

\bibitem[DV75a]{Donsker75b}
M.~D. Donsker and S.~R.~S. Varadhan, \emph{Asymptotic evaluation of certain
  {M}arkov process expectations for large time. {I}. {II}}, Comm. Pure Appl.
  Math. \textbf{28} (1975), 1--47; ibid. 28 (1975), 279--301. \MR{0386024 (52
  \#6883)}

\bibitem[DV75b]{Donsker1975AsymptoticEvaluationWienerIntegrals}
\bysame, \emph{Asymptotic evaluation of certain {W}iener integrals for large
  time}, Functional integration and its applications ({P}roc. {I}nternat.
  {C}onf., {L}ondon, 1974), Clarendon Press, Oxford, 1975, pp.~15--33.
  \MR{0486395 (58 \#6141)}

\bibitem[DV75c]{Donsker75VariationalFormula}
\bysame, \emph{On a variational formula for the principal eigenvalue for
  operators with maximum principle}, Proc. Nat. Acad. Sci. U.S.A. \textbf{72}
  (1975), 780--783. \MR{0361998 (50 \#14440)}

\bibitem[DV76a]{Donsker76}
\bysame, \emph{Asymptotic evaluation of certain {M}arkov process expectations
  for large time. {III}}, Comm. Pure Appl. Math. \textbf{29} (1976), no.~4,
  389--461. \MR{0428471 (55 \#1492)}

\bibitem[DV76b]{Donsker1976PrincipalEigenvalue}
\bysame, \emph{On the principal eigenvalue of second-order elliptic
  differential operators}, Comm. Pure Appl. Math. \textbf{29} (1976), no.~6,
  595--621. \MR{0425380 (54 \#13336)}

\bibitem[DV83]{Donsker1983}
\bysame, \emph{Asymptotic evaluation of certain {M}arkov process expectations
  for large time. {IV}}, Comm. Pure Appl. Math. \textbf{36} (1983), no.~2,
  183--212. \MR{690656 (84h:60128)}

\bibitem[DVJ08]{Daley2008}
D.~J. Daley and D.~Vere-Jones, \emph{An introduction to the theory of point
  processes. {V}ol. {II}}, second ed., Springer, New York, 2008. \MR{2371524
  (2009b:60150)}

\bibitem[DZ98]{Dembo98}
A.~Dembo and O.~Zeitouni, \emph{Large deviations techniques and applications},
  second ed., Springer-Verlag, New York, 1998. \MR{1619036 (99d:60030)}

\bibitem[EN00]{Engel2000}
K.-J. Engel and R.~Nagel, \emph{One-parameter semigroups for linear evolution
  equations}, Springer-Verlag, New York, 2000. \MR{1721989 (2000i:47075)}

\bibitem[Fri75]{Friedman1975}
A.~Friedman, \emph{Stochastic differential equations and applications. {V}ol.
  1}, Academic Press [Harcourt Brace Jovanovich Publishers], New York, 1975.
  \MR{0494490 (58 \#13350a)}

\bibitem[FS06]{Fleming2006}
W.~H. Fleming and H.~M. Soner, \emph{Controlled {M}arkov processes and
  viscosity solutions}, second ed., Springer, New York, 2006. \MR{2179357
  (2006e:93002)}

\bibitem[GT83]{Gilbarg83}
D.~Gilbarg and N.~S. Trudinger, \emph{Elliptic partial differential equations
  of second order}, second ed., Springer-Verlag, Berlin, 1983. \MR{737190
  (86c:35035)}

\bibitem[JKO94]{Jikov1994}
V.~V. Jikov, S.~M. Kozlov, and O.~A. Oleinik, \emph{Homogenization of
  differential operators and integral functionals}, Springer-Verlag, Berlin,
  1994. \MR{1329546 (96h:35003b)}

\bibitem[Kos07]{Kosygina07}
E.~Kosygina, \emph{Homogenization of stochastic {H}amilton-{J}acobi equations:
  brief review of methods and applications}, Stochastic analysis and partial
  differential equations, Contemp. Math., vol. 429, Amer. Math. Soc., 2007,
  pp.~189--204. \MR{2391536 (2009h:60113)}

\bibitem[KRV06]{Kosygina2006}
E.~Kosygina, F.~Rezakhanlou, and S.~R.~S. Varadhan, \emph{Stochastic
  homogenization of {H}amilton-{J}acobi-{B}ellman equations}, Comm. Pure Appl.
  Math. \textbf{59} (2006), no.~10, 1489--1521. \MR{2248897 (2009d:35017)}

\bibitem[KS91]{Karatzas1991}
I.~Karatzas and S.~E. Shreve, \emph{Brownian motion and stochastic calculus},
  second ed., Springer-Verlag, New York, 1991. \MR{1121940 (92h:60127)}

\bibitem[LS05]{Lions2005}
P.-L. Lions and P.~E. Souganidis, \emph{Homogenization of ``viscous''
  {H}amilton-{J}acobi equations in stationary ergodic media}, Comm. Partial
  Differential Equations \textbf{30} (2005), no.~1-3, 335--375. \MR{2131058
  (2005k:35019)}

\bibitem[Mou12]{Mourrat11}
J.-C. Mourrat, \emph{Lyapunov exponents, shape theorems and large deviations
  for the random walk in random potential}, ALEA Lat. Am. J. Probab. Math.
  Stat. \textbf{9} (2012), 165--209. \MR{2923190}

\bibitem[Pin95]{Pinsky1995}
R.~G. Pinsky, \emph{Positive harmonic functions and diffusion}, Cambridge
  University Press, Cambridge, 1995. \MR{1326606 (96m:60179)}

\bibitem[RASY13]{Rassoul11QuenchedFreeEnergy}
F.~Rassoul-Agha, T.~Sepp{\"a}l{\"a}inen, and A.~Yilmaz, \emph{Quenched free
  energy and large deviations for random walks in random potentials}, Comm.
  Pure Appl. Math. \textbf{66} (2013), no.~2, 202--244. \MR{2999296}

\bibitem[Roc70]{Rockafellar70}
R.~T. Rockafellar, \emph{Convex analysis}, Princeton University Press,
  Princeton, N.J., 1970. \MR{0274683 (43 \#445)}

\bibitem[Rud91]{Rudin1991Functional}
W.~Rudin, \emph{Functional analysis}, second ed., McGraw-Hill Inc., New York,
  1991. \MR{1157815 (92k:46001)}

\bibitem[Rue12]{Ruess2012}
J.~Rue{\ss}, \emph{Lyapunov exponents of {B}rownian motion: Decay rates for
  scaled {P}oissonian potentials and bounds}, Markov Process. Related Fields
  \textbf{18} (2012), no.~4, 595--612. \MR{3051654}

\bibitem[Rue13]{Ruess2013}
J.~Rue\ss, \emph{A variational formula for the {L}yapunov exponent of
  {B}rownian motion in stationary ergodic potential}, ArXiv e-prints (2013).

\bibitem[{Rue}14]{Ruess2014}
J.~{Rue{\ss}}, \emph{Continuity results and estimates for the {L}yapunov
  exponent of {B}rownian motion in random potential}, ArXiv e-prints (2014).

\bibitem[Sch74]{Schaefer1974}
H.~H. Schaefer, \emph{Banach lattices and positive operators}, Springer-Verlag,
  New York, 1974. \MR{0423039 (54 \#11023)}

\bibitem[Sch88]{Schroeder88}
C.~Schroeder, \emph{Green's functions for the {S}chr\"odinger operator with
  periodic potential}, J. Funct. Anal. \textbf{77} (1988), no.~1, 60--87.
  \MR{930391 (89g:35033)}

\bibitem[Sch97]{Schechter1997}
E.~Schechter, \emph{Handbook of analysis and its foundations}, Academic Press
  Inc., San Diego, CA, 1997. \MR{1417259 (98b:00009)}

\bibitem[Sch09]{Schmitz2009}
T.~Schmitz, \emph{On the equivalence of the static and dynamic points of view
  for diffusions in a random environment}, Stochastic Process. Appl.
  \textbf{119} (2009), no.~8, 2501--2522. \MR{2532210 (2011a:60290)}

\bibitem[Sio58]{Sion1958}
M.~Sion, \emph{On general minimax theorems}, Pacific J. Math. \textbf{8}
  (1958), 171--176. \MR{0097026 (20 \#3506)}

\bibitem[Sto01]{Stollmann01}
P.~Stollmann, \emph{Caught by disorder}, Birkh\"auser Boston Inc., Boston, MA,
  2001. \MR{1935594 (2004a:82048)}

\bibitem[Szn94]{Sznitman94}
A.-S. Sznitman, \emph{Shape theorem, {L}yapounov exponents, and large
  deviations for {B}rownian motion in a {P}oissonian potential}, Comm. Pure
  Appl. Math. \textbf{47} (1994), no.~12, 1655--1688. \MR{1303223 (96b:60217)}

\bibitem[Szn98]{Sznitman98}
\bysame, \emph{Brownian motion, obstacles and random media}, Springer-Verlag,
  Berlin, 1998. \MR{1717054 (2001h:60147)}

\bibitem[Wil70]{Willard1970}
S.~Willard, \emph{General topology}, Addison-Wesley Publishing Co., Reading,
  Mass.-London-Don Mills, Ont., 1970. \MR{0264581 (41 \#9173)}

\bibitem[Zer98]{Zerner98}
M.~P.~W. Zerner, \emph{Directional decay of the {G}reen's function for a random
  nonnegative potential on {${\bf Z}^d$}}, Ann. Appl. Probab. \textbf{8}
  (1998), no.~1, 246--280. \MR{1620370 (99f:60172)}

\end{thebibliography}

\def\cprime{$'$} \def\cprime{$'$}
\providecommand{\bysame}{\leavevmode\hbox to3em{\hrulefill}\thinspace}
\providecommand{\MR}{\relax\ifhmode\unskip\space\fi MR }
\providecommand{\MRhref}[2]{%
  \href{http://www.ams.org/mathscinet-getitem?mr=#1}{#2}
}
\providecommand{\href}[2]{#2}

\end{document}